\documentclass[11pt,reqno,twoside]{article}

\usepackage{cite}
\usepackage{graphicx}
\usepackage{amssymb}
\usepackage{epstopdf}
\usepackage{float}
\usepackage{placeins}

\usepackage{booktabs} 
\usepackage{array} 
\usepackage{paralist} 
\usepackage{verbatim} 
\usepackage{subfigure} 
\usepackage{tabularx}
\usepackage{amsmath,amsfonts,amsthm,mathrsfs,amssymb,cite}
\usepackage[usenames]{color}
\usepackage{bm}

\newtheorem{thm}{Theorem}[section]

\newtheorem{lem}{Lemma}[section]

\theoremstyle{definition}

\theoremstyle{remark}

\newtheorem{rem}{Remark}[section]
\numberwithin{equation}{section}

\newcommand{\bu}{\bm{u}}

\newcommand{\bGa}{\bm{\Gamma}}
\newcommand{\bx}{\bm{x}}

\newcommand{\by}{\bm{y}}
\newcommand{\rmi}{\mathrm{i}}

\newcommand{\bz}{\bm{z}}

\newcommand{\Lcal}{\mathcal{L}}

\usepackage[hidelinks]{hyperref}

\title{A Non-Decoupled Time-Domain Direct Sampling Method for Inverse Elastic Medium Scattering}

\begin{document}

\author{
		Lefu Cai\footnote{ School of Mathematics, Harbin Institute of Technology, Harbin, People’s Republic of China. (25B312009@stu.hit.edu.cn).}
		\and
		Hongjie Li\footnote{Yau Mathematical Sciences Center, Tsinghua University, Beijing, China. The work of this author was substantially supported by NSFC grant (12401561). (hongjieli@tsinghua.edu.cn; hongjie\_li@yeah.net).}
		\and
		Xianchao Wang\footnote{School of Mathematics, Harbin Institute of Technology, Harbin, People’s Republic of China. The work of this author was supported by NSFC grant 12471397 and Heilongjiang
Provincial Natural Science Foundation grant YQ2024A003. (xcwang90@gmail.com).}
	}

    \date{}
    \maketitle

\begin{abstract}
This work is concerned with an inverse medium problem for elastic waves, in which unknown inhomogeneities are reconstructed from time-resolved boundary measurements.
We propose a novel time-domain direct sampling method for locating scatterers from a single incident source, without imposing specific assumptions on the temporal profile of the excitation.
In particular, the imaging functional introduces a time-shifted correlation strategy that replaces the traditional $P$-$S$ wave decomposition with a travel-time alignment mechanism, thereby enabling direct imaging from the coupled elastic wave field.
To analyze the proposed time-domain imaging functional, we employ Parseval's identity for the Fourier--Laplace transform and reformulate the functional in the frequency domain. By exploiting properties of modified Bessel functions, we characterize the asymptotic behavior of the imaging functional and show that it attains its maximum at the target location, which enables reliable identification of the scatterer.
Rigorous theoretical justifications are provided to substantiate the effectiveness of the proposed method.
Numerical experiments are also presented to demonstrate its performance and applicability.
\medskip


\noindent{\bf 2020 Mathematics Subject Classification:}~~

\end{abstract}

\maketitle

\section{Introduction}\label{sec:1}

In this paper, we investigate an inverse medium scattering problem in time-domain elasticity, where the objective is to reconstruct unknown scatterers from boundary measurements of coupled elastic waves.
Typically, an incident elastic wave is emitted toward the targets of interest, and an array of receivers is placed on a closed or open measurement boundary located away from the scatterers. The scattered elastic wave fields recorded by these receivers are then used to determine the locations and geometric shapes of the unknown objects.
This class of inverse problems arises in numerous scientific and engineering applications, including seismic exploration in geophysics \cite{cheng2026feature}, nondestructive evaluation of engineering structures \cite{ ammari2018super}, and medical ultrasound elastography \cite{bal2015displacement}. 

We next present the mathematical formulation of the inverse medium problem for time-dependent elastic waves. Let $ D \subset \mathbb{R}^d$, $d=2,3$, denote the inhomogeneous scatterers with Lam\'e parameters $(\lambda_1,\mu_1)$ and density $\rho_1(\bm{x})$.
The background medium $ \mathbb{R}^d\setminus\overline{D} $ is characterized by the constants $(\lambda_2,\mu_2,\rho_2)$, and we assume that $\lambda_2/\lambda_1=\mu_2/\mu_1$.
Then the global parameters $(\lambda,\mu, \rho)$ admit the representation
\begin{equation}\label{eq:confi}
	\begin{split}
		\lambda(\bm{x}) &= \lambda_1\chi(D) + \lambda_2\chi(\mathbb{R}^d\backslash\overline{D}), \\
		\mu(\bm{x})&=\mu_1\chi(D) + \mu_2\chi(\mathbb{R}^d\backslash\overline{D}), \\
		\rho(\bm{x}) &= \rho_1(\bm x)\chi(D) + \rho_2\chi(\mathbb{R}^d\backslash\overline{D}). \\
	\end{split}
\end{equation}
Moreover, the Lam\'{e} constants in the two regions are assumed to satisfy the following strong convexity condition:
 \begin{equation*}
  \mathrm{i)}.~~\mu_i>0\qquad\mbox{and}\qquad \mathrm{ii)}.~~3\lambda_i+2\mu_i>0,
 \end{equation*}
with $i=1,2$.
Given a causal incident wave $\bm u^i$, namely $\bm u^i \equiv 0$ for $t\leq 0$, the propagation of the elastic scattered wave $\bm{u}^s(\bm{x},t)$ is governed by the following initial-value problem:
\begin{equation}\label{eq:system}
	\begin{aligned}
		&\mathcal{L}_{\lambda,\mu}\bm{u}^s(\bm{x},t)-\rho(\bm{x})\frac{\partial^2 \bm{u}^s(\bm{x},t)}{\partial t^2}=\left(\rho(\bm{x})-\frac{\lambda(\bm{x})}{\lambda_2} \rho_2 \right)\frac{\partial^2 \bm{u}^i(\bm{x},t)}{\partial t^2},\quad (\bm{x},t)\in \mathbb{R}^d\times\mathbb{R}_+,\\
		&\bm{u}^s(\bm{x},0)=\partial_t \bm{u}^s(\bm{x},0)=\bm{0},\quad \bm{x}\in \mathbb{R}^d.
	\end{aligned}
\end{equation}
In \eqref{eq:system}, the differential operator $\mathcal{L}_{\lambda,\mu}$, associated with the Lam\'{e} parameters $\lambda$ and $\mu$, is defined by
\begin{equation*}
\mathcal{L}_{\lambda,\mu}\bm{u}:=\mu\triangle \bm{u} +(\lambda+\mu)\nabla\nabla\cdot \bm{u}.
\end{equation*}
The inverse problem considered in this work is to determine the compact support of the inhomogeneous medium from measurements collected on an observation surface \(\Gamma \subset \mathbb{R}^d \setminus \overline{D}\), namely,
\begin{equation*}
	\Lambda := \left\{ \bm{u}(\bm{x},t) : (\bm{x},t) \in \Gamma \times \mathbb{R}_+ \right\}.
\end{equation*}

During the past few decades, inverse medium problems in elasticity have been studied extensively in the frequency domain. In particular, uniqueness and stability results have been established in the literature \cite{Hahner_1993, Elschner_2010, doi:10.1137/22M1508546}. Nevertheless, the numerical reconstruction of unknown scatterers remains highly challenging because of the inherent nonlinearity and ill-posedness of the inverse problem.
Existing numerical methods can be broadly classified into iterative and non-iterative approaches. Iterative methods reformulate the inverse problem as a PDE-constrained optimization problem, typically using shape optimization, level-set representations, or domain derivative techniques \cite{BAO2018263, Li_2016}. Although these approaches can yield high-resolution reconstructions, they require repeated solutions of the full forward elastic wave system, which leads to substantial computational costs, especially in higher dimensions. Moreover, their performance depends strongly on the quality of the initial guess.
To address these limitations, non-iterative methods have been developed as efficient alternatives. These approaches avoid repeated forward simulations and enable fast imaging. Among them, sampling-type methods provide direct characterizations of the scatterer support, including the linear sampling method \cite{Colton_2019}, the factorization method \cite{kirsch2007factorization}, the probe method \cite{potthast2006survey}, the direct sampling method \cite{ito2013direct}, and the recently developed monotonicity method \cite{harrach2013monotonicity, Eberle-Blick_2023}. These methods rely on operator-theoretic criteria to determine whether sampling points or regions lie inside the scatterer \cite{Hu_2013, Arens_2001}. It should be noted, however, that the aforementioned imaging approaches are restricted to single-frequency measurements.
Furthermore, in the context of elastic waves, an additional difficulty arises from the coexistence of compressional (P) and shear (S) waves, which are usually decoupled before mode-dependent imaging functionals are constructed \cite{Ji_2018, doi:10.1137/19M1237788}.

Compared with frequency-domain data, time-domain signals are often more naturally available and physically informative in practical scenarios \cite{ma2024imaging, jin2025iterative, klibanov2025convexification}. However, most existing time-domain methods have been developed for scalar acoustic or electromagnetic wave models governed by the Helmholtz or Maxwell equations \cite{sini2022inverse, lahivaara2022time}. Extending these methods to elastodynamics governed by the Navier equation is far from straightforward. The primary difficulty lies in the coexistence of P- and S-waves, which propagate at different speeds and are intrinsically coupled through the governing system. To address this challenge, one approach is to employ the Helmholtz decomposition together with retarded layer potentials to derive coupled boundary integral formulations in the Laplace domain \cite{Zhao_2022}. Nevertheless, this reconstruction procedure still relies on iterative solvers and requires explicit separation of wave modes, thereby introducing additional computational and modeling complexity in the time domain.
Consequently, considerable effort has been devoted to time-domain sampling-type approaches, leading to the development of methods such as the time-domain linear sampling method \cite{chen2010sampling, Guo_2013} and time-domain factorization techniques \cite{cakoni2019factorization, haddar2020time}. It is important to note that both classes of methods encounter theoretical challenges associated with transmission eigenvalue problems. Although the solvability of the time-domain linear sampling method has been rigorously established for acoustic waves \cite{cakoni2021analysis}, its extension to elastic waves remains open. To overcome these difficulties, the direct sampling method has recently been extended to the time domain \cite{doi:10.1137/23M1622854,doi:10.1137/24M1701071}. The main idea is to construct a space--time imaging functional involving delayed measurements and carefully designed test functions. The behavior of the corresponding indicator function is then characterized by exploiting properties of modified Bessel functions.
However, to the best of our knowledge, no direct sampling method has yet been proposed for the time-domain inverse scattering problem governed by the Navier equation. Moreover, the signals recorded by the sensors contain superposed P- and S-wave components. This naturally raises the important question of whether a non-decoupled imaging functional can be designed directly from the elastic wave field.

Motivated by this gap, we develop a new direct sampling framework for inverse elastic medium scattering problems in the time domain. Let $\widetilde{D}$ denote the sampling domain containing the scatterer, so that $\overline{D} \subset \widetilde{D}$.
We introduce the imaging functional
\begin{equation}\label{eq:time-indicator}
	\begin{aligned}
		\mathcal{I}(\bm{z})=&\int_{-\infty}^{\infty} \Big|\sum_{\tau=p,s}\int_{ \Gamma} \frac{ e^{-\sigma(t+c_{\tau}^{-1}|\bm{x}-\bm{z}|)}}{|\bm{x}-\bm{z}|^{\frac{d-1}{2}}} \bm{\Gamma}_{\tau}(\bm{x}-\bm{z})\bm{u}^s(\bm{x}, t+c_{\tau}^{-1}|\bm{x}-\bm{z}|) \,\mbox{d}s(\bm{x})\Big|^2\mbox{d}t,
	\end{aligned}
\end{equation}
for $\bm{z}\in \widetilde{D}$, where the constants $c_p=\sqrt{(\lambda_2+2\mu_2)/\rho_2}$ and $c_s=\sqrt{\mu_2/\rho_2}$ represent the phase velocities of the compressional and shear waves, respectively, and
\begin{equation}\label{eq:degaps}
	\bGa_p(\bm{x}) = -\hat{\bm{x}}\hat{\bm{x}}^\top, \quad \bGa_s(\bm{x}) = -\mathbf{I} + \hat{\bm{x}}\hat{\bm{x}}^\top. 
\end{equation}
It follows from \eqref{eq:time-indicator} that the proposed functional is a non-decoupled time-domain indicator. It incorporates contributions from both P- and S-waves while accounting for their respective propagation delays, and it constructs the image by applying the corresponding time-shift corrections directly to the measured scattered field.
To analyze the behavior of the imaging functional, we first establish the well-posedness of the forward time-domain elastic wave system. We then employ the Fourier--Laplace transform to derive an equivalent frequency-domain representation of the indicator function.
The imaging mechanism of the corresponding frequency-domain indicator is further investigated through a quasi-static analysis and properties of modified Bessel functions. Numerical experiments are finally presented to demonstrate the robustness and resolution capability of the proposed method.

The main contributions of this work can be summarized as follows. First, we propose a time-domain imaging functional for inverse elastic scattering, thereby extending the frequency-domain direct sampling methodology to the time domain. Second, the proposed method does not require decoupling of the P- and S-wave components, which substantially simplifies both the theoretical analysis and the practical implementation. Third, the reconstruction functional involves only space--time integrals over the measurement boundary and does not require iterative forward solves, making it particularly suitable for large-scale and real-time imaging problems in elastodynamics.

The remainder of this paper is organized as follows. Section 2 establishes the analytical framework through the Fourier--Laplace transform and studies the well-posedness of the associated system. Section 3 analyzes the behavior of the proposed imaging functional. In Section 4, we present numerical experiments that validate the theoretical results and demonstrate the effectiveness of the proposed method. A non-decoupled frequency-domain imaging functional for the elastic model is presented in the Appendix.

\section{Fourier--Laplace Transform and Forward Problem Analysis}

In this section, we recall the Fourier--Laplace transform and then discuss the well-posedness of system \eqref{eq:system}. We first introduce the notation for the relevant space--time Sobolev spaces and the Fourier--Laplace transform.
Let $X$ be a Hilbert space. We denote by $\mathcal{D}(\mathbb{R},X)$ the space of $X$-valued $C_0^\infty$ functions on $\mathbb{R}$ with compact support in $(-\infty,\infty)$, and by $\mathcal{D}'(\mathbb{R},X)$ the corresponding space of $X$-valued distributions. The Schwartz space of $X$-valued functions on the real line is denoted by $\mathcal{S}(\mathbb{R},X)$, and $\mathcal{S}'(\mathbb{R},X)$ denotes the corresponding space of tempered distributions. With this notation, we define

\begin{equation*}
\mathcal{L}'_\sigma(\mathbb{R},X):=\left\{f\in \mathcal{D}'(\mathbb{R},X):e^{-\sigma t} f\in \mathcal{S}'(\mathbb{R},X)\right\},\quad \sigma\in\mathbb{R}.
\end{equation*}
Let $\mathbb{C}_{\sigma_0} =\{\omega=\xi+\text{i}\sigma\in\mathbb{C}: \Im(\omega)\geq\sigma_0>0\}$ denote a half-plane in the complex plane.
To pass from the time domain to the frequency domain, we use the Fourier--Laplace transform of a function $f\in\mathcal{L}'_\sigma(\mathbb{R},X)$, defined by
\begin{equation*}
\mathcal{L}[f](\omega):=\int_{-\infty}^{+\infty}e^{\text{i} \omega t}f(t) \, \text{d}t,\quad \omega\in \mathbb{C}_{\sigma_0}.
\end{equation*}
For the Fourier--Laplace transform, the following Parseval identity holds:
\begin{equation}\label{eq:parseval}
\int_{-\infty}^{+\infty}|e^{-\sigma t}f(t)|^2\,\text{d}t= \frac{1}{2\pi} \int_{-\infty+\text{i}\sigma}^{+\infty+\text{i}\sigma} |\mathcal{L}[f](\omega)|^2 \, \text{d}\omega.
\end{equation}
To handle time-dependent wave fields, we introduce the following Hilbert space for $m\in\mathbb{R}$:
\begin{equation*}
H_\sigma^m(\mathbb{R},X):=\Big\{f\in\mathcal{L}'_\sigma(\mathbb{R},X) :\int_{-\infty+\text{i}\sigma}^{+\infty+\text{i}\sigma} |\omega|^{2m}\left\|\mathcal{L}[f](\omega)\right\|^2_X\text{d}\omega<+\infty\Big\},
\end{equation*}
equipped with the norm
\begin{equation*}
\left\|f\right\|_{H_\sigma^m(\mathbb{R},X)}= \Big(\int_{-\infty+\text{i}\sigma}^{+\infty+\text{i}\sigma} |\omega|^{2m}\left\|\mathcal{L}[f](\omega)\right\|^2_X\text{d}\omega\Big)^{1/2}.
\end{equation*}
The wave fields can therefore be transformed from the time domain to the frequency domain through the Fourier--Laplace transform, namely,
\begin{equation*}
\hat{\bm{u}}^s(\bm{x},\omega):=\int_{-\infty}^{+\infty} e^{\text{i}\omega t} \bm{u}^s(\bm{x},t)\, \text{d}t,\quad
\hat{\bm{u}}(\bm{x},\omega):=\int_{-\infty}^{+\infty} e^{\text{i}\omega t} \bm{u}(\bm{x},t) \, \text{d}t.
\end{equation*}
System \eqref{eq:system} can then be expressed in the frequency domain as
\begin{equation}\label{eq:frequency-domain}
    \mathcal{L}_{\lambda_2,\mu_2}\hat{\bm{u}}^s(\bm{x},\omega)+ \rho_2 \omega^2\hat{\bm{u}}^s(\bm{x},\omega)= \Big(\rho_2- \frac{\lambda_2}{\lambda(\bm{x})}\rho(\bm{x})\Big) \omega^2 \hat{\bm{u}}(\bm{x},\omega),\quad \bm{x}\in \mathbb{R}^d.
\end{equation}
By the Lippmann--Schwinger integral representation and a straightforward simplification, the solution to \eqref{eq:frequency-domain} can be written as
\begin{equation}\label{eq:usxomega}
\hat{\bm{u}}^s(\bm{x},\omega)=\int_{D}\bm{\Gamma}^\omega(\bm{x},\bm{y}) \left(\frac{\lambda_2}{\lambda_1}\rho_1(\bm{y})-\rho_2\right) \omega^2 \hat{\bm{u}}(\bm{y},\omega) \, \mbox{d}\bm{y}.
\end{equation}
In \eqref{eq:usxomega}, the function $\bGa^{\omega}=(\Gamma^{\omega}_{i,j})_{i,j=1}^d$ denotes the fundamental solution of the operator $\Lcal_{\lambda_2,\mu_2} + \rho_2\omega^2$ and is given by \cite{L0069}
\begin{equation}\label{eq:Gamma}
\bm{\Gamma}^{\omega}_{i,j}(\bm{x})=
    \left\{
    \begin{aligned}
   & -\frac{\delta_{ij}\mathrm{i}}{4\mu}H_0^{(1)}(\omega c_s^{-1}|\bm{x}|) + \frac{\mathrm{i}}{4 \omega^2\rho_2}\partial_{ij}^2\left(H_0^{(1)}(\omega c_p^{-1}|\bm{x}|)-H_0^{(1)}(\omega c_s^{-1}|\bm{x}|)\right) , \ d=2,\\
    &-\frac{\delta_{ij}}{4\pi\mu|\bm{x}|}e^{\rmi \omega c_s^{-1} |\bm{x}|} + \frac{1}{4\pi \omega^2\rho_2}\partial_{ij}^2 \frac{e^{\rmi \omega c_p^{-1}|\bm{x}|} - e^{\rmi \omega c_s^{-1}|\bm{x}|}}{|\bm{x}|}, \qquad \qquad \qquad  \quad \ d=3,
    \end{aligned}\right.
    \end{equation}
where $H_0^{(1)}$ is the Hankel function of the first kind of order zero. Furthermore, as $|\bm{x}|\rightarrow \infty$, the following asymptotic expansion holds:
\begin{equation*}
\bm{\Gamma}^{\omega}(\bm{x})=
    a_{d,1}\bm{\Gamma}_p(\bm{x} )\frac{e^{\rmi\omega c_p^{-1}|\bm{x} |}}{|\omega|^{\frac{3-d}{2}}|\bm{x} |^{\frac{d-1}{2}}} + a_{d,2}\bm{\Gamma}_s(\bm{x} )\frac{e^{\rmi\omega c_s^{-1}|\bm{x} |}}{\omega^{\frac{3-d}{2}}|\bm{x} |^{\frac{d-1}{2}}} +\mathcal{O}\big(\frac{1}{|\bm{x} |^{\frac{d+1}{2}}}\big),
\end{equation*}
where
\begin{equation}\label{eq:defa14}
a_{2,1}=\frac{1+\mathrm{i}}{4\sqrt{\pi(\lambda_2+2\mu_2)}},\quad a_{2,2}=\frac{1+\mathrm{i}}{4\sqrt{\pi\mu_2}},\quad
a_{3,1}=\frac{1}{4\pi(\lambda_2+2\mu_2)},\quad a_{3,2}=\frac{1}{4\pi\mu_2}.
\end{equation}

We next discuss the well-posedness of \eqref{eq:system}. To this end, we first recall the following lemma; further details can be found in \cite{Lubich_1994} and \cite[Chapter 3]{Sayas_2016}.

\begin{lem}\label{lem:regularity}
    Let $X$ and $Y$ be Banach spaces, and let $\mathcal{B}(X,Y)$ denote the space of all bounded linear operators from $X$ to $Y$. Suppose that $\psi:\mathbb{C}_{\sigma_0}\rightarrow \mathcal{B}(X,Y)$ is an analytic function satisfying the bound
    \begin{equation*}
    \|\psi(\omega)\|_{\mathcal{B}(X,Y)}\leq C|\omega|^k,\quad \omega\in \mathbb{C}_{\sigma_0},\quad k\in \mathbb{R},
    \end{equation*}
    where $C>0$ is a constant independent of $\omega$. Define the inverse transform
    \begin{equation*}
    \check{\psi}(t):=\frac{1}{2\pi}\int_{-\infty+\mathrm{i}\sigma}^{\infty+\mathrm{i}\sigma} e^{-\mathrm{i}\omega t}\psi(\omega) \,\mathrm{d}\omega, 
    \end{equation*}
    and the associated convolution operator
    \begin{equation*}
    \Psi g:=\int_{-\infty}^{\infty}\check{\psi}(\ell)g(t-\ell) \,\mathrm{d}\ell.
    \end{equation*}
    Then, for any $m\in\mathbb{R}$, the operator $\Psi$ extends to a bounded operator from $H_{\sigma}^{m+k}(\mathbb{R},X)$ into $H_{\sigma}^{m}(\mathbb{R},Y)$.
    \end{lem}

Using Lemma \ref{lem:regularity}, we establish the following result.
\begin{thm}\label{thm:regularity}
Suppose that the medium parameters are given by \eqref{eq:confi} and that the density $\rho_1(\bm{x})$ has a positive lower bound. Let $m\in\mathbb{R}$ and $\sigma\geq\sigma_0>0$. Then, for any incident field $\bm{u}^i(\bm{x},t)\in H_{\sigma}^{m}(\mathbb{R},(L^2(D))^d)$, system \eqref{eq:system} admits a unique scattered-field solution $\bm{u}^s(\bm{x},t)\in H_{\sigma}^{m-2}(\mathbb{R},(H^1(D))^d)$. Moreover, the solution $\bm{u}^s(\bm{x},t)$ satisfies the following regularity estimate:
    \begin{equation*}
    \|\bm{u}^s\|_{H_{\sigma}^{m-2}(\mathbb{R},(H^1(D))^d)}\leq C \|\bm{u}^i\|_{H_{\sigma}^{m}(\mathbb{R},(L^2(D))^d)},
    \end{equation*}
    where $C>0$ is a constant independent of $\bm{u}^i$.
    \end{thm}
    \begin{proof}
    For $\bx\in D$, equation \eqref{eq:frequency-domain} can be written as
    \begin{equation*}
    \mathcal{L}_{\lambda_1,\mu_1}\hat{\bm{u}}^s(\bm{x},\omega)+ \rho_1(\bm{x}) \omega^2\hat{\bm{u}}^s(\bm{x},\omega)= \left(\frac{\lambda_1}{\lambda_2}\rho_2-\rho_1(\bm{x})\right)\omega^2 \hat{\bm{u}}^i(\bm{x},\omega).
    \end{equation*}
    Multiplying the equation by a smooth test function $\bm{v}\in(C_0^\infty(D))^d$ and integrating over $D$ yield
    \begin{equation*}
    \int_D \left(\mathcal{L}_{\lambda_1,\mu_1}\hat{\bm{u}}^s(\bm{x},\omega)+ \rho_1(\bm{x}) \omega^2\hat{\bm{u}}^s(\bm{x},\omega)\right) \overline{\bm{v}}(\bm{x})\, \mathrm{d}\bm{x} = \int_D \left(\frac{\lambda_1}{\lambda_2}\rho_2-\rho_1(\bm{x})\right)\omega^2 \hat{\bm{u}}^i(\bm{x},\omega) \overline{\bm{v}}(\bm{x}) \,\mathrm{d}\bm{x}.
    \end{equation*}
    Applying integration by parts, we obtain
    \begin{equation}\label{eq:variationA}
    A(\hat{\bm{u}}^s,\bm{v})= \int_D \left(\frac{\lambda_1}{\lambda_2}\rho_2-\rho_1(\bm{x})\right)\omega^2 \hat{\bm{u}}^i(\bm{x},\omega) \overline{\bm{v}}(\bm{x})\, \mathrm{d}\bm{x},
    \end{equation}
    where the sesquilinear form $A$ is defined by
    \begin{equation}\label{eq:Auv}
    A(\bm{u},\bm{v}):=\int_D \mu_1 \nabla \bm{u}:\nabla \overline{\bm{v}} + (\lambda_1+\mu_1) (\nabla\cdot \bm{u})(\nabla\cdot \overline{\bm{v}}) - \rho_1\omega^2 \bm{u}\cdot\overline{\bm{v}} \, \mathrm{d}\bm{x},
    \end{equation}
    with the notation
    \begin{equation*}
    \nabla \bm{u}:\nabla \overline{\bm{v}}:=\sum_{i,j} \frac{\partial u_i}{\partial x_j}\frac{\partial \overline{v}_i}{\partial x_j}.
    \end{equation*}
    Since the density $\rho_1(\bm{x})$ is bounded, i.e., $0<\rho_{\mathrm{min}} \leq \rho_1(\bm{x})\leq\rho_{\mathrm{max}}$, setting $\bm{v}=\bm{u}$ in \eqref{eq:Auv} yields
    \begin{equation*}
    \begin{aligned}
    \Re \left(\mathrm{i} \overline{\omega} A (\bm{u},\bm{u})\right)=&\Re \left( \int_D \mathrm{i}\mu_1 \overline{\omega}\nabla \bm{u}:\nabla \overline{\bm{u}} + (\lambda_1+\mu_1) (\nabla\cdot \bm{u})(\nabla\cdot \overline{\bm{u}}) - \mathrm{i}\rho_1\omega|\omega|^2 \bm{u}\cdot \overline{\bm{u}}\,\mathrm{d}\bm{x}\right)\\
    =&\Im(\omega) \int_D \mu_1 \nabla \bm{u}:\nabla \overline{\bm{u}} + (\lambda_1+\mu_1) (\nabla\cdot \bm{u})(\nabla\cdot \overline{\bm{u}}) + \rho_1|\omega|^2 \bm{u}\cdot \overline{\bm{u}}\,\mathrm{d}\bm{x}\\
    \geq &\Im(\omega) \min\left\{\mu_1,\rho_{\mathrm{min}}|\omega|^2\right\} \|\bm{u}\|^2_{(H^1(D))^d}.
    \end{aligned}
    \end{equation*}
    On the other hand, the Cauchy--Schwarz inequality gives the continuity estimate
    \begin{equation*}
    \Re \left(\mathrm{i} \overline{\omega} A (\bm{u},\bm{v})\right)\leq \Im(\omega) \max\left\{\lambda_1+2\mu_1,\rho_{\mathrm{max}} |\omega|^2\right\} \|\bm{u}\|_{(H^1(D))^d} \|\bm{v}\|_{(H^1(D))^d}.
    \end{equation*}
    Therefore, the sesquilinear form $A(\cdot,\cdot)$ is both coercive and continuous. By the Lax--Milgram theorem, there exists a unique solution $\hat{\bm{u}}^s$ satisfying \eqref{eq:variationA}, and consequently \eqref{eq:frequency-domain}. Furthermore, we observe that
    \begin{equation*}
    \begin{aligned}
    &\Re \left(\mathrm{i} \overline{\omega} \int_D \left(\frac{\lambda_1}{\lambda_2}\rho_2-\rho_1(\bm{x})\right)\omega^2 \hat{\bm{u}}^i(\bm{x},\omega) \overline{\bm{v}}(\bm{x})\, \mathrm{d}\bm{x}\right)\\
   & \quad  \leq\Im(\omega)|\omega|^2 \left\|\frac{\lambda_1}{\lambda_2}\rho_2-\rho_1(\bm{x})\right\|_{C(D)} \|\hat{\bm{u}}^i\|_{(L^2(D))^d} \|\bm{v}\|_{(L^2(D))^d}.
    \end{aligned}
    \end{equation*}
    Taking $\bm{v}=\bm{u}^s$ in \eqref{eq:variationA}, we obtain
    \begin{equation*}
    \begin{aligned}
    \Im(\omega) \min&\left\{\mu_1,\rho_{\mathrm{min}}|\omega|^2\right\} \|\hat{\bm{u}}^s\|^2_{(H^1(D))^d}\leq \Re \left(\mathrm{i} \overline{\omega} A (\hat{\bm{u}}^s, \hat{\bm{u}}^s)\right)\\
    &=\Re \left(\mathrm{i} \overline{\omega} \int_D \left(\frac{\lambda_1}{\lambda_2}\rho_2-\rho_1(\bm{x})\right)\omega^2 \hat{\bm{u}}^i(\bm{x},\omega) \overline{\hat{\bm{u}}^s}(\bm{x},\omega) \,  \mathrm{d}\bm{x}\right)\\
    &\leq \Im(\omega)|\omega|^2 \left\|\frac{\lambda_1}{\lambda_2}\rho_2-\rho_1(\bm{x})\right\|_{C(D)} \|\hat{\bm{u}}^i\|_{(L^2(D))^d} \|\hat{\bm{u}}^s\|_{(L^2(D))^d}.
    \end{aligned}
    \end{equation*}
    Consequently, for $\sigma\geq\sigma_0>0$, we obtain the following frequency-domain estimate:
    \begin{equation*}
    \begin{aligned}
    \|\hat{\bm{u}}^s\|_{(H^1(D))^d}\leq& \max \left\{\frac{|\omega|^2}{\mu_1},\frac{1}{\rho_{\mathrm{min}}}\right\} \left\|\frac{\lambda_1}{\lambda_2}\rho_2-\rho_1(\bm{x})\right\|_{C(D)} \|\hat{\bm{u}}^i\|_{(L^2(D))^d}\\
    \leq &\max \left\{\frac{|\omega|^2}{\rho_{\mathrm{min}}\sigma_0^2} , \frac{|\omega|^2}{\mu_1},\frac{1}{\rho_{\mathrm{min}}}\right\} \left\|\frac{\lambda_1}{\lambda_2}\rho_2-\rho_1(\bm{x})\right\|_{C(D)} \|\hat{\bm{u}}^i\|_{(L^2(D))^d}\\
    =& |\omega|^2 \max\left\{\frac{1}{\rho_{\mathrm{min}}\sigma_0^2}, \frac{1}{\mu_1}\right\} \left\|\frac{\lambda_1}{\lambda_2}\rho_2-\rho_1(\bm{x})\right\|_{C(D)} \|\hat{\bm{u}}^i\|_{(L^2(D))^d}.
    \end{aligned}
    \end{equation*}
    
    Define the frequency-domain solution operator $\hat{K}(\omega)$ associated with \eqref{eq:frequency-domain} by
    \begin{equation*}
    \hat{K}(\omega):\hat{\bm{u}}^i\in (L^2(D))^d \mapsto \hat{\bm{u}}^s\in (H^1(D))^d,
    \end{equation*}
    and define the corresponding time-domain solution operator $\mathcal{K}$ for \eqref{eq:system} by
    \begin{equation*}
    \mathcal{K}:\bm{u}^i\in  H_{\sigma}^{m}(\mathbb{R},(L^2(D))^d) \mapsto \bm{u}^s \in  H_{\sigma}^{m-2}(\mathbb{R},(H^1(D))^d).
    \end{equation*}
    The operator $\mathcal{K}$ can be expressed through convolution as
    \begin{equation*}
    \mathcal{K}\bm{u}^i=\bm{u}^s= \mathcal{L}^{-1}\left[\hat{\bm{u}}^s\right]= \mathcal{L}^{-1}\left[\hat{K} \hat{\bm{u}}^i\right]= \mathcal{L}^{-1}\left[\mathcal{L}[K]\mathcal{L}[\bm{u}^i] \right]= K*\bm{u}^i.
    \end{equation*}
    Finally, applying Lemma \ref{lem:regularity} with $k=2$ yields the time-domain regularity estimate
    \begin{equation*}
    \|\bm{u}^s\|_{H_{\sigma}^{m}(\mathbb{R},(H^1(D))^d)} \leq C \|\bm{u}^i\|_{H_{\sigma}^{m-2}(\mathbb{R},(L^2(D))^d)}.
    \end{equation*}
    \end{proof}

\section{Analysis of the Time-Domain Direct Sampling Method}
In this section, we analyze the asymptotic behavior of the proposed imaging functional \eqref{eq:time-indicator} and clarify its mechanism for identifying the locations of the scatterers. To this end, we employ the Fourier--Laplace transform and conduct the theoretical analysis in the frequency domain.

\begin{thm}
    Let the time-domain imaging functional be defined by \eqref{eq:time-indicator}. Applying the Fourier--Laplace transform, the imaging functional admits the following frequency-domain representation:
    \begin{equation*}
    \mathcal{I}(\bm{z})=\frac{1}{2\pi}\int_{-\infty+\rmi\sigma}^{\infty+\rmi\sigma}\Big| \int_{\Gamma} \sum_{\tau=p,s} \bm{\Gamma}_\tau(\bm{x}-\bm{z}) \frac{e^{-\rmi\Re (\omega) c_\tau^{-1}|\bm{x}-\bm{z}|}}{|\bm{x}-\bm{z}|^{\frac{d-1}{2}}} \hat{\bm{u}}^s(\bm{x},\omega) \, \mathrm{d}s(\bm{x})\Big|^2\mathrm{d}\omega,\quad \bm{z}\in \widetilde{D}.
    \end{equation*}
    \end{thm}
    \begin{proof}
    Applying Parseval's identity \eqref{eq:parseval} and performing a direct calculation, we obtain
    \begin{align*}
     \mathcal{I}(\bm{z})=&\int_{-\infty}^{\infty} \Big|\int_{\Gamma} \sum_{\tau=p,s} \frac{  e^{-\sigma(t+c_\tau^{-1}|\bm{x}-\bm{z}|)}}{|\bm{x}-\bm{z}|^{\frac{d-1}{2}}} \bm{\Gamma}_\tau(\bm{x}-\bm{z})\bm{u}^s(\bm{x}, t+c_\tau^{-1}|\bm{x}-\bm{z}|)\,\mbox{d}s(\bm{x})\Big|^2\mbox{d}t\\
    =&\frac{1}{2\pi}\int_{-\infty+\text{i}\sigma}^{\infty+\text{i}\sigma}\Big| \int_{\Gamma} \sum_{\tau=p,s} \frac{e^{-\sigma c_\tau^{-1}|\bm{x}-\bm{z}|}}{|\bm{x}-\bm{z}|^{\frac{d-1}{2}}} \bm{\Gamma}_\tau(\bm{x}-\bm{z})\mathcal{L}[\bm{u}^s(\bm{x}, t+c_\tau^{-1}|\bm{x}-\bm{z}|)](\omega)\,\mbox{d}s(\bm{x})\Big|^2\mbox{d}\omega\\
    =&\frac{1}{2\pi}\int_{-\infty+\text{i}\sigma}^{\infty+\text{i}\sigma}\Big| \int_{\Gamma}\sum_{\tau=p,s}\bm{\Gamma}_\tau(\bm{x}-\bm{z}) \frac{e^{-\text{i}\Re (\omega) c_\tau^{-1}|\bm{x}-\bm{z}|}}{|\bm{x}-\bm{z}|^{\frac{d-1}{2}}}  \hat{\bm{u}}^s(\bm{x},\omega) \, \mbox{d}s(\bm{x})\Big|^2\mbox{d}\omega.
    \end{align*}
    
    This completes the proof.
    \end{proof}
    
    Before investigating the asymptotic behavior of the proposed imaging functional, we introduce several preliminary lemmas.
    
    \begin{lem}\label{lem:gamma_ps}
        The functions $\bGa_p$ and $\bGa_s$ defined in \eqref{eq:degaps} satisfy the following properties:
        \[
       \bGa_p(\bm{x})\bGa_p(\bm{x})=-\bGa_p(\bm{x}), \quad \bGa_s(\bm{x})\bGa_s(\bm{x})=-\bGa_s(\bm{x}), \quad \bGa_p(\bm{x})\bGa_s(\bm{x})=\bGa_s(\bm{x})\bGa_p(\bm{x})=0.
        \]
       \end{lem}
       \begin{proof}
        The proof follows from a direct calculation using the definitions of $\bGa_p(\bm{x})$ and $\bGa_s(\bm{x})$ in \eqref{eq:degaps}, and is therefore omitted.
       \end{proof}

       \begin{lem}\label{lem:integral}
        For $\bm{\alpha}\in \mathbb{C}^3$, the following identities hold:
        \begin{equation*}
    \begin{aligned}
    &\int_{ \mathbb{S}^2} \mathbf{I} e^{\hat{\bx}\cdot\bm{\alpha}} \,{\rm d}s(\hat{\bx}) = 4\pi i_0( \sqrt{\bm{\alpha}\cdot\bm{\alpha}})\mathbf{I},\\
    &\int_{ \mathbb{S}^2} \hat{\bm{x}}\hat{\bm{x}}^\top e^{\hat{\bx}\cdot\bm{\alpha}} \,{\rm d}s(\hat{\bm{x}}) = \frac{4\pi}{3}(i_0(\sqrt{\bm{\alpha}\cdot\bm{\alpha}})- i_2(\sqrt{\bm{\alpha}\cdot\bm{\alpha}}))\mathbf{I} +4\pi i_2( \sqrt{\bm{\alpha}\cdot\bm{\alpha}})\frac{\bm{\alpha}\bm{\alpha}^\top } {\bm{\alpha}\cdot\bm{\alpha}}.
    \end{aligned}
    \end{equation*}
        Furthermore, in the two-dimensional case, for $\bm{\alpha}\in \mathbb{C}^2$, one has
        \begin{equation*}
        \begin{aligned}
        &\int_{ \mathbb{S}^1} \mathbf{I} e^{\hat{\bx}\cdot\bm{\alpha}} \,{\rm d}s(\hat{\bx})= 2\pi I_0( \sqrt{\bm{\alpha}\cdot\bm{\alpha}})\mathbf{I},\\
        &\int_{ \mathbb{S}^1}\hat{\bm{x}}\hat{\bm{x}}^\top  e^{\hat{\bx}\cdot\bm{\alpha}} \,{\rm d}s(\hat{\bx})= \pi(I_0(\sqrt{\bm{\alpha}\cdot\bm{\alpha}})- I_2(\sqrt{\bm{\alpha}\cdot\bm{\alpha}}))\mathbf{I} +2\pi I_2( \sqrt{\bm{\alpha}\cdot\bm{\alpha}})\frac{\bm{\alpha}\bm{\alpha}^\top } {\bm{\alpha}\cdot\bm{\alpha}}.
        \end{aligned}
        \end{equation*}
        Here, $I_n$ and $i_n$ denote, respectively, the modified Bessel function and the modified spherical Bessel function of the first kind of order $n$, and $\mathbb{S}^{d-1}$ denotes the unit sphere in $\mathbb{R}^d$ for $d=2,3$.
        \end{lem}
    \begin{proof}
        It suffices to prove the result in the three-dimensional case, since the two-dimensional case follows by an analogous argument. Let $Y_n^m(\hat{\bm{x}})$ denote the spherical harmonic functions, where $\hat{\bm{x}}=(\hat{x}_1,\hat{x}_2,\hat{x}_3)\in\mathbb{S}^2$. We recall their explicit forms:
        \begin{equation}\label{eq:harmonic}
        \begin{aligned}
        &Y_0^0(\hat{\bm{x}})=\frac{1}{2}\sqrt{\frac{1}{\pi}},\quad Y_2^0(\hat{\bm{x}})=\frac{1}{4}\sqrt{\frac{5}{\pi}}(3\hat{x}_3^2-1)= \frac{1}{4}\sqrt{\frac{5}{\pi}}(2-3\hat{x}_1^2-3\hat{x}_2^2),\\ 
        &Y_2^{-1}(\hat{\bm{x}})-Y_2^1(\hat{\bm{x}})= \sqrt{\frac{15}{2\pi}} \hat{x}_1\hat{x}_3,\quad Y_2^{-1}(\hat{\bm{x}})+Y_2^1(\hat{\bm{x}})= -\rmi\sqrt{\frac{15}{2\pi}} \hat{x}_2\hat{x}_3,\\ 
        &Y_2^{-2}(\hat{\bm{x}})-Y_2^2(\hat{\bm{x}})= -\rmi\sqrt{\frac{15}{2\pi}} \hat{x}_1\hat{x}_2,\quad Y_2^{-2}(\hat{\bm{x}})+Y_2^2(\hat{\bm{x}})= \sqrt{\frac{15}{8\pi}}(\hat{x}_1^2-\hat{x}_2^2).
        \end{aligned}
        \end{equation}
    We expand the exponential function in spherical harmonics as
    \begin{equation*}
    e^{\hat{\bx}\cdot\bm{\alpha}}=4\pi \sum_{n=1}^{\infty}\sum_{m=-n}^n i_n(\sqrt{\bm{\alpha}\cdot\bm{\alpha}})Y_n^m(\frac{\bm{\alpha}} {\sqrt{\bm{\alpha}\cdot\bm{\alpha}}}) \overline{Y_n^m(\hat{\bx}) }.
    \end{equation*}
    Multiplying by spherical harmonic functions and integrating over the unit sphere yield
    \begin{equation*}
    \int_{ \mathbb{S}^2} Y_n^m(\hat{\bx}) e^{\hat{\bx}\cdot\bm{\alpha}}\mbox{d}s(\hat{\bx}) = 4\pi i_n( \sqrt{\bm{\alpha}\cdot\bm{\alpha}}) Y_n^m(\frac{\bm{\alpha}} {\sqrt{\bm{\alpha}\cdot\bm{\alpha}}}).
    \end{equation*}
    Combining the above identities, we obtain
    \begin{equation*}
    \begin{aligned}
    &\int_{ \mathbb{S}^2} \mathbf{I} e^{\hat{\bx}\cdot\bm{\alpha}} \, \mbox{d}s(\hat{\bx}) = 4\pi i_0( \sqrt{\bm{\alpha}\cdot\bm{\alpha}})\mathbf{I},\\
    &\int_{ \mathbb{S}^2} \hat{\bm{x}}\hat{\bm{x}}^\top e^{\hat{\bx}\cdot\bm{\alpha}} \, \mbox{d}s(\hat{\bm{x}}) = \frac{4\pi}{3}(i_0(\sqrt{\bm{\alpha}\cdot\bm{\alpha}})- i_2(\sqrt{\bm{\alpha}\cdot\bm{\alpha}}))\mathbf{I} +4\pi i_2( \sqrt{\bm{\alpha}\cdot\bm{\alpha}})\frac{\bm{\alpha}\bm{\alpha}^\top } {\bm{\alpha}\cdot\bm{\alpha}}.
    \end{aligned}
    \end{equation*}
    This completes the proof.
    \end{proof}
    
    \begin{rem}
    In the two-dimensional case, the conclusion can be established using the expansion
    \begin{equation*}
    e^{\hat{\bx}\cdot\bm{\alpha}}=I_0(\sqrt{\bm{\alpha}\cdot\bm{\alpha}})+ 2\sum_{n=1}^{\infty} I_n( \sqrt{\bm{\alpha}\cdot\bm{\alpha}})\cos(n\varphi),
    \end{equation*}
    where $\varphi$ denotes the angle between $\bm{x}$ and $\bm{\alpha}$.
    \end{rem}

    Building on the identities established in Lemma \ref{lem:integral}, we derive the following key estimates for the subsequent analysis of the indicator function.

    \begin{lem}\label{lem:Gomega}
    Let $\Gamma=\{\bm{x}\in \mathbb{R}^d\backslash \overline{D}:|\bm{x}|=r\}$ denote the sphere of radius $r$. For $\by\in D$, define the matrix-valued integral
    \begin{equation*}
    \mathbf{G}_\omega(\bm{z}, \by)
    =\int_{\Gamma} \sum_{\tau=p,s} \bm{\Gamma}_\tau(\bm{x}-\bm{z}) \frac{e^{-\rmi\Re (\omega) c_\tau^{-1}|\bm{x}-\bm{z}|}}{|\bm{x}-\bm{z}|^{\frac{d-1}{2}}}  \bm{\Gamma}^\omega(\bm{x},\bm{y}) \, \mathrm{d}s(\bm{x}).
    \end{equation*}
    Assume that $\omega\in \mathbb{C}_{\sigma_0}$ satisfies $\omega\in(0+\mathrm{i}\sigma, \xi_{\mathrm{max}}+\mathrm{i}\sigma)$, where $\xi_{\mathrm{max}}$ is a positive constant such that $|\Re(\omega)|<\xi_{\mathrm{max}}$. Then the following asymptotic behavior holds as $r\rightarrow \infty$:
    \begin{equation}\label{eq:Gomegai}
    |\mathbf{G}^{i,j}_\omega(\bm{z}, \bm{y})| = 4^{d-2}\pi\left|a_{d,1} \frac{e^{-\sigma c_p^{-1}r}}{|\omega|^{\frac{3-d}{2}}} \mathbf{K}_p^{i,j}(\bm{\eta}) +a_{d,2} \frac{ e^{-\sigma c_s^{-1}r}}{|\omega|^{\frac{3-d}{2}}} \mathbf{K}_s^{i,j}(\bm{\eta}) \right| \Big\{1+\mathcal{O}\big(\frac{|\omega|}{r^{\frac{d-1}{2}}}\big)\Big\},
    \end{equation}
    where $\bm{\eta}=\sigma \bm{y}+\mathrm{i}\Re(\omega)(\bm{z}-\bm{y})$, and
    \begin{equation}\label{eq:Keta}
    \begin{aligned}
     \mathbf{K}_{p}(\bm{\eta})&=\left\{
    \begin{aligned}
    -\left(I_0(c_{p}^{-1}\sqrt{\bm{\eta}\cdot\bm{\eta}})-I_2( c_{p}^{-1}\sqrt{\bm{\eta}\cdot\bm{\eta}})\right)\mathbf{I} -2I_2( c_{p}^{-1}\sqrt{\bm{\eta}\cdot\bm{\eta}})\frac{\bm{\eta}\bm{\eta}^\top} {\sqrt{\bm{\eta}\cdot\bm{\eta}}}, &\quad d=2,\\
    -\frac{i_0(c_{p}^{-1}\sqrt{\bm{\eta}\cdot\bm{\eta}}) -i_2(c_{p}^{-1}\sqrt{\bm{\eta}\cdot\bm{\eta}})}{3}\mathbf{I} -i_2( c_{p}^{-1}\sqrt{\bm{\eta}\cdot\bm{\eta}})\frac{\bm{\eta}\bm{\eta}^\top} {\sqrt{\bm{\eta}\cdot\bm{\eta}}}, & \quad d=3 ,
    \end{aligned}\right.\\
    \mathbf{K}_{s}(\bm{\eta})&=\left\{
    \begin{aligned}
    -\left(I_0(c_{s}^{-1}\sqrt{\bm{\eta}\cdot\bm{\eta}})+I_2( c_{s}^{-1}\sqrt{\bm{\eta}\cdot\bm{\eta}})\right)\mathbf{I} +2I_2( c_{s}^{-1}\sqrt{\bm{\eta}\cdot\bm{\eta}})\frac{\bm{\eta}\bm{\eta}^\top} {\sqrt{\bm{\eta}\cdot\bm{\eta}}}, &\quad d=2,\\
    -\frac{2i_0(c_{s}^{-1}\sqrt{\bm{\eta}\cdot\bm{\eta}}) +i_2(c_{s}^{-1}\sqrt{\bm{\eta}\cdot\bm{\eta}})}{3}\mathbf{I} +i_2( c_{s}^{-1}\sqrt{\bm{\eta}\cdot\bm{\eta}})\frac{\bm{\eta}\bm{\eta}^\top} {\sqrt{\bm{\eta}\cdot\bm{\eta}}}, & \quad d=3.
    \end{aligned}\right.
    \end{aligned}
    \end{equation}
   
     Furthermore, the components $\mathbf{G}^{i,j}_\omega( \bz, \by)$ satisfy the estimates
    \begin{equation}\label{eq:Gomegae}
  |\mathbf{G}^{i,j}_\omega(\bm{z},\bm{y})|= \big(a_{d,1}e^{-\sigma c_p^{-1}r}+ a_{d,2}e^{-\sigma c_s^{-1}r}\big)\mathcal{O}\Big(\frac{1}{|\omega||\bm{z}-\bm{y}|^{\frac{d-1}{2}}}\Big) \left\{1+\mathcal{O}\big(\frac{|\omega|}{r^{\frac{d-1}{2}}}\big)\right\},
    \end{equation}
    as $|\bm{z}-\bm{y}|\rightarrow \infty$ and $r\rightarrow \infty$, where the constants $a_{d,1}$ and $a_{d,2}$ are defined in \eqref{eq:defa14}.
    \end{lem}
    \begin{proof}
    We give the proof for the three-dimensional case only, since the two-dimensional case follows by an analogous argument. We first note the following asymptotic expansions:
    \begin{equation*}
    |\bm{x}-\bm{y}|=|\bm{x}|-\hat{\bm{x}}\cdot \by+\mathcal{O}\big(\frac{1}{|\bm{x}|}\big),\quad \frac{1}{|\bm{x}-\bm{y}|}=\frac{1}{|\bm{x}|} \Big\{1+\mathcal{O}\big(\frac{1}{|\bm{x}|}\big)\Big\}.
    \end{equation*}
    Consequently, by invoking Lemma \ref{lem:gamma_ps}, we obtain
    \begin{equation*}
    \mathbf{G}_\omega(\bm{z},\bm{y})=-\left(a_{3,1}\mathbf{G}_{\omega,p}(\bm{z},\bm{y}) + a_{3,2}\mathbf{G}_{\omega,s}(\bm{z},\bm{y})\right) \Big\{1+\mathcal{O}\big(\frac{|\omega|}{r}\big)\Big\}, \quad r\rightarrow \infty,
    \end{equation*}
    where
    \begin{equation*}
    \mathbf{G}_{\omega,\tau}(\bm{z},\bm{y})
     =e^{-\sigma c_\tau^{-1}r}\int_{\mathbb{S}^2} \bm{\Gamma}_{\tau}(\bm{x}) e^{ c_{\tau}^{-1}\hat{\bm{x}}\cdot \bm{\eta}} \, \mbox{d}s(\hat{\bm{x}}), \quad \tau=p,s, 
    \end{equation*}
    and $\bm{\eta}=\sigma \bm{y}+\mathrm{i}\Re(\omega)(\bm{z}-\bm{y})$. Applying Lemma \ref{lem:integral}, it follows that
    \begin{align*}
    \mathbf{G}_{\omega,p}(\bm{z},\bm{y})=& 4\pi e^{-\sigma c_{p}^{-1}r} \Big(-\frac{i_0(c_{p}^{-1}\sqrt{\bm{\eta}\cdot\bm{\eta}}) -i_2(c_{p}^{-1}\sqrt{\bm{\eta}\cdot\bm{\eta}})}{3}\mathbf{I} -i_2( c_{p}^{-1}\sqrt{\bm{\eta}\cdot\bm{\eta}})\frac{\bm{\eta}\bm{\eta}^\top} {\sqrt{\bm{\eta}\cdot\bm{\eta}}}\Big),\\
     \mathbf{G}_{\omega,s}(\bm{z},\bm{y})=& 4\pi e^{-\sigma c_{s}^{-1}r} \Big(-\frac{2i_0(c_{s}^{-1}\sqrt{\bm{\eta}\cdot\bm{\eta}}) +i_2(c_{s}^{-1}\sqrt{\bm{\eta}\cdot\bm{\eta}})}{3}\mathbf{I} +i_2( c_{s}^{-1}\sqrt{\bm{\eta}\cdot\bm{\eta}})\frac{\bm{\eta}\bm{\eta}^\top} {\sqrt{\bm{\eta}\cdot\bm{\eta}}}\Big).
    \end{align*}
    This establishes \eqref{eq:Gomegai}. Furthermore, using properties of oscillatory integrals \cite{stein1993real}, we have
    \begin{equation*}
    \Big|\int_{\mathbb{S}^{d-1}} e^{\rmi\Re (\omega)c_{\tau}^{-1}\bm{x}\cdot(\bm{z}-\bm{y})} e^{\sigma c_{\tau}^{-1}\bm{x}\cdot \bm{y}}\,\text{d}s(\bm{x})\Big| =\mathcal{O}\Big(\frac{1}{|\omega(\bm{z}-\bm{y})|^{\frac{d-1}{2}}}\Big), \quad |\bm{z}-\bm{y}|\rightarrow\infty ,
    \end{equation*}
Thus, the estimate \eqref{eq:Gomegae} follows, which completes the proof.
    \end{proof}

    Assume that the bounded domain $D$, which characterizes the region containing the inhomogeneous scatterers, consists of finitely many well-separated subdomains with small diameters. Specifically, let
    \begin{equation}\label{eq:Dscatters}
    D=\bigcup_{j=1}^N D_j,\quad D_j=\bm{y}^j+\varepsilon B_j, \quad N\in\mathbb{N},
    \end{equation}
    where $\bm{y}^j$ and $B_j$ denote, respectively, the center and the normalized support of the scatterer $D_j$, while $\varepsilon>0$ is a scaling parameter. Furthermore, define the minimum separation distance between the scatterers by
    \begin{equation*}
    L:=\min_{1\leq i,j\leq N,i\neq j} \text{dist}(D_i,D_j),
    \end{equation*}
    where $\text{dist}(D_i,D_j)$ denotes the distance between $D_i$ and $D_j$. Under these geometric assumptions, the Born approximation remains valid.
    
    \begin{lem}\label{lem:us}
        Suppose that the parameters are given by \eqref{eq:confi} and that the density $\rho_1(\bm{x})$ is bounded in $D$. Assume that the scatterer configuration \eqref{eq:Dscatters} holds. Then the solution $\hat{\bm{u}}^s(\bm{x},\omega)$ to \eqref{eq:frequency-domain} satisfies the following asymptotic estimates:
        \begin{equation*}
        \|\hat{\bm{u}}^s(\cdot,\omega)\|_{C(D)^d}
        =\left\{
        \begin{aligned}
        &\mathcal{O}\left(|\omega|^2\varepsilon^2 \ln(|\omega| \varepsilon)\right), &d=2,\\ &\mathcal{O}\left(|\omega|^2\varepsilon^2\right), &d=3, 
        \end{aligned}
        \right.
        \end{equation*}
        as $|\omega|\varepsilon\rightarrow 0$.
        \end{lem}
        \begin{proof}
        In accordance with \eqref{eq:usxomega}, introduce the operator $T:C(D)^d\rightarrow C(D)^d$ defined by
        \begin{equation*}
        T\bm{f}(\bm{x})= \omega^2\int_{D}\bm{\Gamma}^\omega(\bm{x},\bm{y}) \Big(\frac{\lambda_2}{\lambda_1}\rho_1(\bm{y})-\rho_2\Big)\bm{f}(\bm{y}) \,\mbox{d}\bm{y}.
        \end{equation*}
        From the definition of $\bm{\Gamma}^\omega(\bm{x},\bm{y})$ in \eqref{eq:Gamma}, it follows that
        \begin{equation*}
        \Big|\int_{D}\bm{\Gamma}_{i,j}^\omega(\bm{x},\bm{y})\,\mbox{d}\bm{y} \Big|\leq\sum_{k=1}^N\Big|\int_{D_k} \bm{\Gamma}_{i,j}^\omega(\bm{x},\bm{y})\,\mbox{d}\bm{y} \Big|=\left\{
        \begin{aligned}
        &\mathcal{O}\left(\varepsilon^2 \ln(|\omega| \varepsilon)\right), &d=2,\\
        &\mathcal{O}\left(\varepsilon^2\right), &d=3.
        \end{aligned}
        \right.
        \end{equation*}
        Consequently, we obtain the following norm estimates:
        \begin{equation*}
         \|T\bm{f}\|_{C(D)^d}
        \leq \Big\|\frac{\lambda_2}{\lambda_1}\rho_1-\rho_2\Big\|_{C(D)^d} \|\bm{f}\|_{C(D)^d}\left\{
        \begin{aligned}
        &\mathcal{O}\left(|\omega|^2\varepsilon^2 \ln(|\omega| \varepsilon)\right), &d=2,\\
        &\mathcal{O}\left(|\omega|^2\varepsilon^2\right), &d=3.
        \end{aligned}
        \right.
        \end{equation*}
        Hence, the operator norm of $T$ satisfies
        \begin{equation*}
        \|T\|=\left\{
        \begin{aligned}
        &\mathcal{O}\left(|\omega|^2\varepsilon^2 \ln(|\omega| \varepsilon)\right), &d=2,\\
        &\mathcal{O}\left(|\omega|^2\varepsilon^2\right), &d=3.
        \end{aligned}
        \right.
        \end{equation*}
        Using the relation $\hat{\bm{u}}^s=T\hat{\bm{u}}$, we deduce that $(I-T)\hat{\bm{u}}=\hat{\bm{u}}^i$. Choosing $|\omega|$ and $\varepsilon$ sufficiently small so that $\|T\|\leq 1/2$ ensures that the operator $I-T$ is invertible. Therefore,
        \begin{equation*}
        \|\hat{\bm{u}}^s\|_{C(D)^d} =\|T\hat{\bm{u}}\|_{C(D)^d} =\|T(I-T)^{-1}\hat{\bm{u}}^i\|_{C(D)^d}= \left\{
        \begin{aligned}
        &\mathcal{O}\left(|\omega|^2\varepsilon^2 \ln(|\omega| \varepsilon)\right), &d=2,\\
        &\mathcal{O}\left(|\omega|^2\varepsilon^2\right), &d=3,
        \end{aligned}
        \right.
        \end{equation*}
        as $|\omega|\varepsilon\rightarrow 0$. 
        This completes the proof.
        \end{proof}

        With the preliminary results established, we now present the principal theorem of this paper, which provides a rigorous analysis of the imaging functional $I(\bz)$.

        \begin{thm}\label{thm:asymptotic}
            Suppose that the conditions of Lemmas \ref{lem:Gomega} and \ref{lem:us} hold, and let the incident wave satisfy $\bm{u}^i(\bm{x},t)\in H^{4+s}_\sigma(\mathbb{R},(L^2(D))^d)$ with $s>0$. Let $\omega_\mathrm{max}:=\xi_\mathrm{max}+\mathrm{i}\sigma \in \mathbb{C}_{\sigma_0}$ denote an upper bound for the frequency of the incident wave $\bm{u}^i(\bm{x},t)$ such that $1<|\omega_\mathrm{max}|\ll1/\varepsilon$. If the sampling point $\bm{z}$ lies in a neighborhood of the scatterer $D_j$, then
            \begin{equation*}
            \begin{aligned}
            \mathcal{I}(\bm{z})=&16^{d-2}\pi^2 \varepsilon^{2d} \big| a_{d,1} e^{-\sigma c_p^{-1}r}\mathbf{K}_p(\bm{\eta}^j)+ a_{d,2} e^{-\sigma c_s^{-1}r}\mathbf{K}_s(\bm{\eta}^j)\big|^2 \Big\{M_j\Big(1+\mathcal{O}\big(\frac{|\omega_\mathrm{max}|} {r^{\frac{d-1}{2}}}\big) +\mathcal{O}(\varepsilon)\Big)+ N_j\mathcal{O}\big(\frac{1}{L^{d-1}}\big)\Big\}\\
            &+\varepsilon^{d} (e^{-2\sigma {c_p^{-1}r}}+e^{-2\sigma {c_s^{-1}r}})\mathcal{O}\big(\frac{1}{|\omega_\mathrm{max}|^{2s-d+3}}\big)\\
            \leq&C_d\varepsilon^d \big( e^{-2\sigma c_p^{-1}r}+  e^{-2\sigma c_s^{-1}r}\big) \Big\{ \varepsilon^d M_j\Big(1+\mathcal{O}\big(\frac{|\omega_\mathrm{max}|}{r^{\frac{d-1}{2}}}\big) +\mathcal{O}(\varepsilon)\Big)+ \varepsilon^d N_j\mathcal{O}\big(\frac{1}{L^{d-1}}\big)+\mathcal{O} \big(\frac{1}{|\omega_\mathrm{max}|^{2s-d+3}}\big)\Big\},
            \end{aligned}
            \end{equation*}
            as $r\rightarrow\infty,\varepsilon\rightarrow0, L\rightarrow\infty$, and $ |\omega_\mathrm{max}|^{2s}\rightarrow \infty$. Here, $\mathbf{K}_p$ and $\mathbf{K}_s$ are defined in \eqref{eq:Keta}, while $M_j$ and $N_j$ are defined by
            \begin{equation}\label{eq:MjNj}
            \begin{aligned}
            M_j&= \int_{0}^{\xi_{\mathrm{max}}+\mathrm{i}\sigma}|\omega|^{d+1}\Big|\int_{B_j} \big(\frac{\lambda_2}{\lambda_1}\rho_1(\bm{y}^j+\varepsilon\bm{\zeta})-\rho_2\big) \hat{\bm{u}}^i(\bm{y}^j+\varepsilon\bm{\zeta},\omega) \,\mathrm{d}\bm{\zeta} \Big|^2 \mathrm{d}\omega,\\
            N_j&= \int_{0}^{\xi_{\mathrm{max}}+\mathrm{i}\sigma}|\omega|^{d-1}\Big|\int_{B_j} \big(\frac{\lambda_2}{\lambda_1}\rho_1(\bm{y}^j+\varepsilon\bm{\zeta})-\rho_2\big) \hat{\bm{u}}^i(\bm{y}^j+\varepsilon\bm{\zeta},\omega)\, \mathrm{d}\bm{\zeta} \Big|^2 \mathrm{d}\omega.
            \end{aligned}
            \end{equation}
            Furthermore, if the sampling point $\bm{z}$ is located sufficiently far from $D$, then
            \begin{equation*}
            \begin{aligned}
            \mathcal{I}(\bm{z})=\varepsilon^{d} \big( e^{-2\sigma c_p^{-1}r}+  e^{-2\sigma c_s^{-1}r}\big) \mathcal{O}\Big(\frac{1}{\mathrm{dist}(\bm{z},D)^{\frac{d-1}{2}}}\Big)     \Big\{\varepsilon^d\Big( 1+\mathcal{O} \big(\frac{|\omega_\mathrm{max}|}{r^{\frac{d-1}{2}}}\big) +\mathcal{O}(\varepsilon)\Big) +\mathcal{O} \big(\frac{1}{|\omega_\mathrm{max}|^{2s-d+3}}\big)\Big\},
            \end{aligned}
            \end{equation*}
            as $r\rightarrow \infty,\varepsilon\rightarrow 0, |\omega_\mathrm{max}|^{2s}\rightarrow \infty$, and $\text{dist}(\bm{z},D)\rightarrow \infty$, where $\mathrm{dist}(\bm{z},D):=\min_{1\leq j\leq N}|\bm{z}-\bm{y}^j|$.
            \end{thm}
        
        \begin{proof}
            It suffices to present the proof in the three-dimensional case, since the two-dimensional analogue follows by a similar argument. By \eqref{eq:usxomega} and Lemma \ref{lem:us}, the scattered wave can be represented as
            \begin{equation*}
            \hat{\bm{u}}^s(\bm{x},\omega)=\omega^2\int_{D}\bm{\Gamma}^\omega(\bm{x},\bm{y}) \big(\frac{\lambda_2}{\lambda_1}\rho_1(\bm{y})-\rho_2\big) \left( \hat{\bm{u}}^i(\bm{y},\omega) +\mathcal{O}(|\omega|^2\varepsilon^2)\right)\mbox{d}\bm{y}.
            \end{equation*}
            Applying a Taylor expansion yields
            \begin{equation*}
            \begin{aligned}
            &\omega^2\int_{D}\bm{\Gamma}^\omega(\bm{x},\bm{y}) \big(\frac{\lambda_2}{\lambda_1}\rho_1(\bm{y})-\rho_2\big) \hat{\bm{u}}^i(\bm{y},\omega) \,\mbox{d}\bm{y}\\
            =&\omega^2\varepsilon^3\sum_{k=1}^N\int_{B_k} \bm{\Gamma}^\omega(\bm{x},\bm{y}^k+\varepsilon\bm{\zeta}) \big(\frac{\lambda_2}{\lambda_1}\rho_1(\bm{y}^k+\varepsilon\bm{\zeta})-\rho_2\big) \hat{\bm{u}}^i(\bm{y}^k+\varepsilon\bm{\zeta},\omega) \,\mbox{d}\bm{\zeta}\\
            =&\omega^2\varepsilon^3\sum_{k=1}^N \bm{\Gamma}^\omega(\bm{x},\bm{y}^k) \Big(\int_{B_k} \big(\frac{\lambda_2}{\lambda_1}\rho_1(\bm{y}^k+\varepsilon\bm{\zeta})-\rho_2\big) \hat{\bm{u}}^i(\bm{y}^k+\varepsilon\bm{\zeta},\omega) \,\mbox{d}\bm{\zeta}+\mathcal{O}(\varepsilon)\Big).
            \end{aligned}
            \end{equation*}
            Assume that $\|\bm{z}-\bm{y}^j\|<L/2$. Using Lemma \ref{lem:Gomega} together with the preceding two identities, we obtain
            \begin{equation*}
            \begin{aligned}
            &\Big|\int_{\Gamma}\sum_{\tau=p,s} \bm{\Gamma}_\tau(\bm{x}-\bm{z}) \frac{e^{-\rmi\Re (\omega) c_\tau^{-1}|\bm{x}-\bm{z}|}}{|\bm{x}-\bm{z}|} \hat{\bm{u}}^s(\bm{x},\omega) \,\mbox{d}s(\bm{x})\Big|\\
            =& \omega^2\varepsilon^3\Big|\sum_{k=1}^N\int_{\Gamma}\sum_{\tau=p,s} \bm{\Gamma}_\tau(\bm{x}-\bm{z}) \frac{e^{-\rmi\Re (\omega) c_\tau^{-1}|\bm{x}-\bm{z}|}}{|\bm{x}-\bm{z}|} \bm{\Gamma}^\omega(\bm{x},\bm{y}^k)\,\mbox{d}s(\bm{x}) \\
            &\qquad \times\ \Big(\int_{B_k} \big(\frac{\lambda_2}{\lambda_1}\rho_1(\bm{y}^k+\varepsilon\bm{\zeta})-\rho_2\big) \hat{\bm{u}}^i(\bm{y}^k+\varepsilon\bm{\zeta},\omega) \,\mbox{d}\bm{\zeta}+\mathcal{O}(\varepsilon|\omega|)\Big)\Big|\\
            =& 4\pi \omega^2\varepsilon^3\Big| \big( a_{3,1} e^{-\sigma c_p^{-1}r}\mathbf{K}_p(\bm{\eta}^j)+ a_{3,2} e^{-\sigma c_s^{-1}r}\mathbf{K}_s(\bm{\eta}^j)\big) \\
            & \qquad\times \int_{B_j} \big(\frac{\lambda_2}{\lambda_1}\rho_1(\bm{y}^j+\varepsilon\bm{\zeta})-\rho_2\big) \hat{\bm{u}}^i(\bm{y}^j+\varepsilon\bm{\zeta},\omega) \,\mbox{d}\bm{\zeta} \Big| \; \Big\{1+\mathcal{O}\big(\frac{|\omega|}{r}\big) +\mathcal{O}\big(\frac{1}{|\omega|L}\big) +\mathcal{O}(\varepsilon)\Big\},
            \end{aligned}
            \end{equation*}
            where $\bm{\eta}^j =-\Re(\omega)(\bm{z}-\bm{y}^j)+\mathrm{i}\sigma \bm{y}^j $ as $r\rightarrow\infty$, $\varepsilon\rightarrow 0$, and $L\rightarrow \infty$. On the one hand,
            \begin{equation}\label{eq:Q1}
            \begin{aligned}
            &\mathbf{Q}_1(\bm{z}):=\frac{1}{\pi}\int_{0}^{\xi_{\text{max}}+\mathrm{i}\sigma}\Big| \int_{\Gamma}\sum_{\tau=p,s} \bm{\Gamma}_\tau(\bm{x}-\bm{z}) \frac{e^{-\mathrm{i}\Re (\omega) c_\tau^{-1}|\bm{x}-\bm{z}|}}{|\bm{x}-\bm{z}|} \hat{\bm{u}}^s(\bm{x},\omega)\, \mbox{d}s(\bm{x})\Big|^2\mbox{d}\omega\\
            = &16\pi^2 \varepsilon^6 \big| a_{3,1} e^{-\sigma c_p^{-1}r}\mathbf{K}_p(\bm{\eta}^j)+ a_{3,2} e^{-\sigma c_s^{-1}r}\mathbf{K}_s(\bm{\eta}^j)\big|^2 \Big\{M_j\Big(1+\mathcal{O}\big(\frac{| \omega_\mathrm{max} | }{r}\big) +\mathcal{O}(\varepsilon)\Big)+ N_j\mathcal{O}\big(\frac{1}{L^2}\big)\Big\},
            \end{aligned}
            \end{equation}
            where $M_j$ and $N_j$ are defined in \eqref{eq:MjNj}.
            The boundedness of $M_j$ and $N_j$ follows from the regularity assumption $\bm{u}^i(\bm{x},t)\in H^{4+s}_\sigma(\mathbb{R},(L^2(D))^3)$ for $s>0$. On the other hand, define
            \begin{equation*}
            \begin{aligned}
            &\mathbf{Q}_2(\bm{z}):=\frac{1}{\pi}\int_{\xi_{\text{max}}+\mathrm{i}\sigma} ^{+\infty+\text{i}\sigma}\Big| \int_{\Gamma}\sum_{\tau=p,s} \bm{\Gamma}_\tau(\bm{x}-\bm{z}) \frac{e^{-\mathrm{i}\Re (\omega) c_\tau^{-1}|\bm{x}-\bm{z}|}}{|\bm{x}-\bm{z}|} 
            \hat{\bm{u}}^s(\bm{x},\omega)\, \mbox{d}s(\bm{x})\Big|^2\mbox{d}\omega\\
            =&\frac{1}{\pi}\int_{\xi_{\text{max}}+\mathrm{i}\sigma} ^{+\infty+\mathrm{i}\sigma}\Big| \omega^2\int_{D}\int_{\Gamma}\sum_{\tau=p,s} \bm{\Gamma}_\tau(\bm{x}-\bm{z}) \frac{e^{-\mathrm{i}\Re (\omega) c_\tau^{-1}|\bm{x}-\bm{z}|}}{|\bm{x}-\bm{z}|} \bm{\Gamma}^\omega(\bm{x},\bm{y}) \Big(\frac{\lambda_2}{\lambda_1}\rho_1(\bm{y})-\rho_2\Big) \hat{\bm{u}}(\bm{y},\omega) \,\mbox{d}s(\bm{x})\,\mbox{d}\bm{y}\Big|^2\mbox{d}\omega. 
            \end{aligned}
            \end{equation*}
            Applying the Cauchy--Schwarz inequality yields
            \begin{equation*}
            \begin{aligned}
            \mathbf{Q}_2(\bm{z})\leq &C\int_{\xi_{\text{max}}+\text{i}\sigma}^{+\infty+\text{i}\sigma} |\omega|^4 \int_D| e^{-\sigma {c_p^{-1}r}}+ e^{-\sigma {c_s^{-1}r}}|^2\,\text{d}\bm{y}\int_D \Big| \Big(\frac{\lambda_2}{\lambda_1}\rho_1(\bm{y})-\rho_2\Big) \hat{\bm{u}}(\bm{y},\omega)\Big|^2 \text{d}\bm{y}\, \mbox{d}\omega\\
            \leq &C\varepsilon^3( e^{-2\sigma {c_p^{-1}r}}+ e^{-2\sigma {c_s^{-1}r}}) \Big\|\frac{\lambda_2}{\lambda_1}\rho_1(\bm{y})-\rho_2\Big\|^2_{C(D)^3} \int_{\xi_{\text{max}}+\text{i}\sigma}^{+\infty+\text{i}\sigma} |\omega|^4\int_D\hat{\bm{u}}^2(\bm{y},\omega)\,\mbox{d}\bm{y}\,\mbox{d}\omega.
            \end{aligned}
            \end{equation*}
            By Theorem \ref{thm:regularity}, the condition $\bm{u}^i(\bm{x},t)\in H^{4+s}_\sigma(\mathbb{R},(L^2(D))^3)$ implies $\bm{u}(\bm{x},t)\in H^{2+s}_\sigma(\mathbb{R},(H^1(D))^3)$. Consequently, there exists a constant $C_1>0$ such that
            \begin{equation*}
            |\omega_\mathrm{max}|^{2s}\int_{\xi_\mathrm{max}+\text{i}\sigma}^{+\infty+\text{i}\sigma} |\omega|^4 \|\hat{\bm{u}}(\bm{y},\omega)\|^2_{(L^2(D))^3}\,\mbox{d}\omega\leq \int_{\xi_\mathrm{max}+\text{i}\sigma}^{+\infty+\text{i}\sigma} |\omega|^{2(2+s)} \|\hat{\bm{u}}(\bm{y},\omega)\|^2_{(H^1(D))^3}\,\mbox{d}\omega\leq C_1.
            \end{equation*}
            Therefore,
            \begin{equation*}
            \mathbf{Q}_2(\bm{z})= \varepsilon^3 (e^{-2\sigma {c_p^{-1}r}}+e^{-2\sigma {c_s^{-1}r}})\mathcal{O}\big(\frac{1}{|\omega_\mathrm{max}|^{2s}}\big), \quad s>0,
            \end{equation*}
            as $|\omega_\mathrm{max}|^{2s}\rightarrow \infty$. Combining \eqref{eq:Q1} with the estimate for $\mathbf{Q_2}(\bm{z})$ and using $\mathcal{I}(\bm{z})=\mathbf{Q}_1(\bm{z})+\mathbf{Q}_2(\bm{z})$, we obtain
            \begin{equation*}
            \begin{aligned}
            \mathcal{I}(\bm{z})\leq\varepsilon^3 \big( e^{-2\sigma c_p^{-1}r}+ e^{-2\sigma c_s^{-1}r}\big) 
            \Big\{\varepsilon^3 M_j\Big(1+\mathcal{O} \big(\frac{|\omega_\mathrm{max}|}{r}\big) +\mathcal{O}(\varepsilon)\Big)+ \varepsilon^3 N_j\mathcal{O}\big(\frac{1}{L^2}\big)+\mathcal{O} \big(\frac{1}{|\omega_{\text{max}}|^{2s}}\big)\Big\},
            \end{aligned}
            \end{equation*}
            as $r\rightarrow\infty$, $\varepsilon\rightarrow0$, $L\rightarrow\infty$, and $ |\omega_\mathrm{max}|^{2s}\rightarrow \infty$. The estimate for a sampling point $\bm{z}$ located far from the domain $D$ follows by a similar argument, yielding
            \begin{equation*}
            \begin{aligned}
            \mathcal{I}(\bm{z})=&\varepsilon^3 \big( e^{-2\sigma c_p^{-1}r}+  e^{-2\sigma c_s^{-1}r}\big) \mathcal{O}\Big(\frac{1}{\text{dist}(\bm{z},D)}\Big)\Big\{\varepsilon^3\Big(1+\mathcal{O}\big(\frac{|\omega_\mathrm{max}|}{r}\big) +\mathcal{O}(\varepsilon)\Big)+\mathcal{O} \big(\frac{1}{|\omega_{\text{max}}|^{2s}}\big)\Big\},
            \end{aligned}
            \end{equation*}
            as $r\rightarrow\infty$, $\varepsilon\rightarrow0$, $|\omega_\mathrm{max}|^{2s}\rightarrow \infty$, and $\text{dist}(\bm{z},D)\rightarrow\infty$, where $\text{dist}(\bm{z},D):=\min_{1\leq j \leq N}|\bm{z}-\bm{y}^j|$.
            
            This completes the proof.
        \end{proof}   
        \begin{rem}
            Theorem \ref{thm:asymptotic} shows that the proposed imaging functional attains a local maximum when the sampling point $\bm{z}$ coincides with a scatterer $D_j$ and decays rapidly as $\bm{z}$ moves away from $D_j$. This property provides a natural mechanism for identifying the scatterer locations through the functional $\mathcal{I}(\bm{z})$. Furthermore, the asymptotic expansion of $\mathcal{I}(\bm{z})$ clarifies the conditions required for accurate reconstruction. Specifically, controlling the higher-order term $\mathcal{O}(|\omega_\mathrm{max}|/r)$ requires a sufficiently large observation radius $r$. The term $\varepsilon|\omega_\mathrm{max}|$ indicates that the scatterer size parameter $\varepsilon$ must be significantly smaller than the reciprocal of the maximum frequency, namely $\varepsilon\ll 1/|\omega_\mathrm{max}|$. In addition, the term $\mathcal{O}(1/L)$ requires a sufficiently large minimum separation distance $L$ between individual scatterers. In summary, high-resolution imaging requires the scatterers to be well separated, small relative to the probing wavelength, and observed from a sufficiently large distance.
        \end{rem}

\section{Numerical Examples}
In this section, we present several numerical examples to illustrate the effectiveness of the proposed imaging functional. To obtain the scattered-field data, we solve the forward problem \eqref{eq:system} using the finite element method.
The unbounded exterior domain is truncated by imposing an absorbing boundary condition. The measurement surface is chosen as a circle, or as a sphere in the three-dimensional case, with radius $r$.

Assume that the elastic wave is generated by a point source located at $\bm{y}\in \mathbb{R}^d$. The incident wave $u^i$ satisfies the following equation:
\begin{equation*}
\mathcal{L}_{\lambda_2,\mu_2}\bm{u}^i(\bm{x},t)-\rho_2\frac{\partial^2 \bm{u}^i(\bm{x},t)}{\partial t^2}=\bm{d}\chi(t)\delta(\bm{x}-\bm{y}),\quad (\bm{x},t)\in \mathbb{R}^d\times\mathbb{R}_+,
\end{equation*}
where $\bm{d}\in \mathbb{R}^d,|\bm{d}|=1$ is a unit direction vector, $\chi$ is a causal temporal signal satisfying $\chi(t)\equiv0$ for $t<0$, and $\delta(\bm{x}-\bm{y})$ is the Dirac delta function. In what follows, we use the Ricker wavelet
\begin{equation}\label{eq:Ricker}
\chi(t)=(1-2a(t-t_0)^2) e^{-a(t-t_0)^2},
\end{equation}
as the source modulation function for the incident field, where $a=2\pi^2 f_0^2$, $f_0$ is the peak frequency of the source, and $t_0$ is the time delay.

\begin{figure}[H]
    \centering
    \begin{minipage}{0.99\textwidth}
        \centering
     \subfigure[computational geometry]{
            \includegraphics[width=0.3\textwidth]{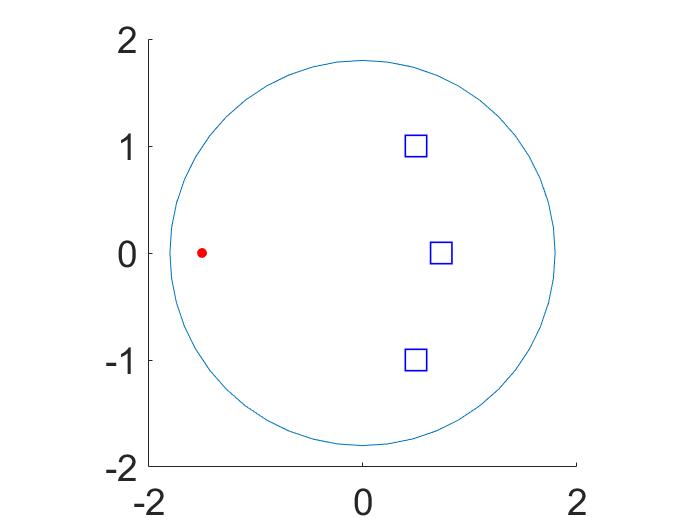}} 
        \subfigure[Ricker wavelet]{
            \includegraphics[width=0.3\textwidth]{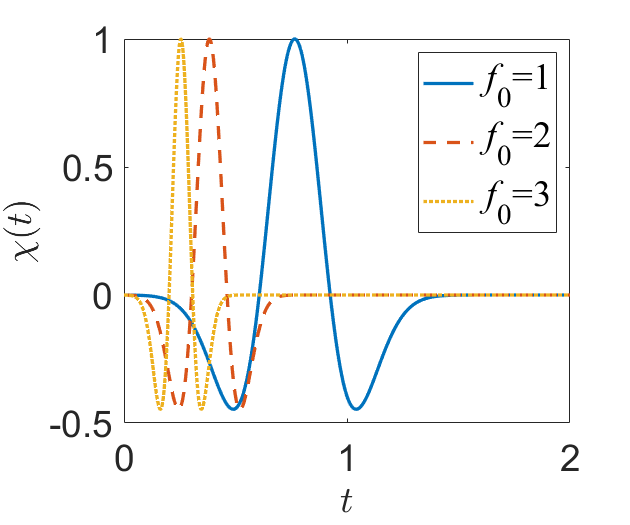}}
        \subfigure[Fourier spectrum of Ricker wavelet]{
            \includegraphics[width=0.3\textwidth]{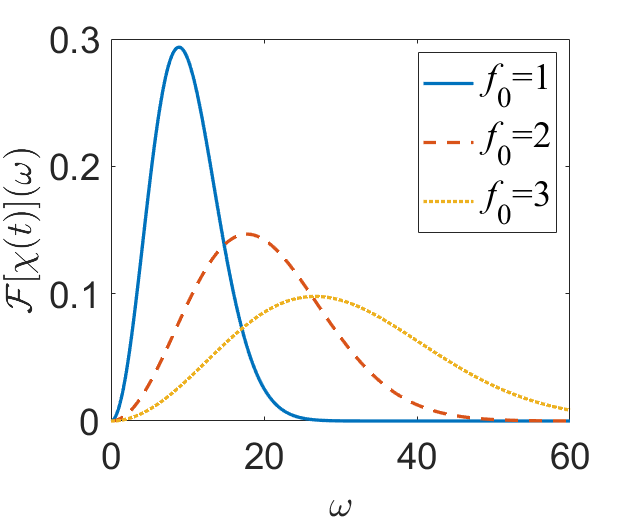}}
        \caption{\label{fig:Ricker}Computational geometry, Ricker wavelet, and corresponding Fourier spectra for different peak frequencies $f_0$.}
    \end{minipage}
    \hfill
    \begin{minipage}{0.99\textwidth}
        \centering
        \subfigure[TFM, $f_0=1$]{
            \includegraphics[width=0.3\textwidth]{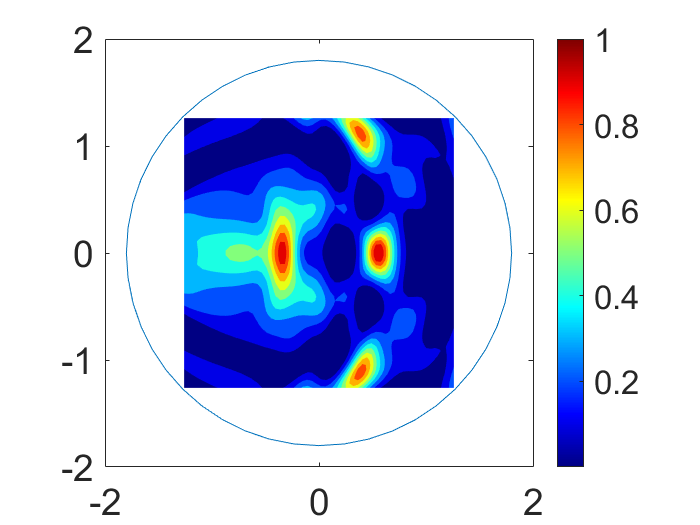}}
        \subfigure[TFM, $f_0=2$]{
            \includegraphics[width=0.3\textwidth]{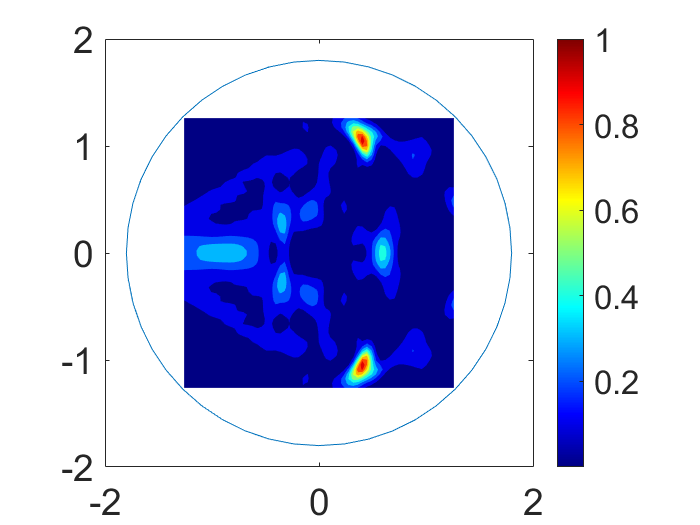}}
        \subfigure[TFM, $f_0=3$]{
            \includegraphics[width=0.3\textwidth]{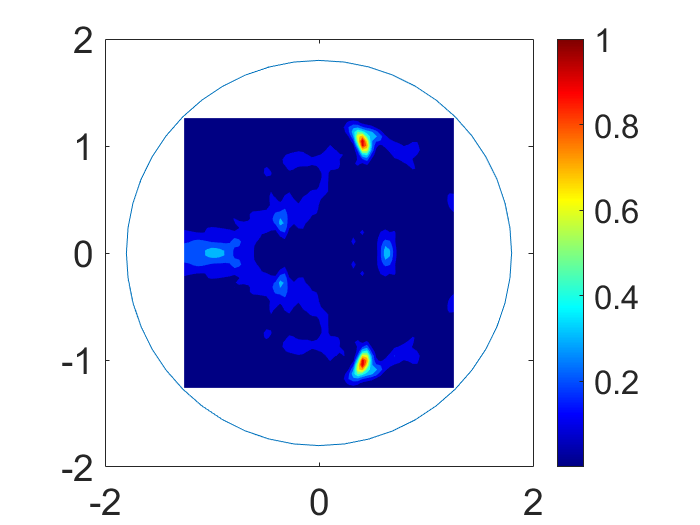}}\\
        \subfigure[DSM, $f_0=1$]{
            \includegraphics[width=0.3\textwidth]{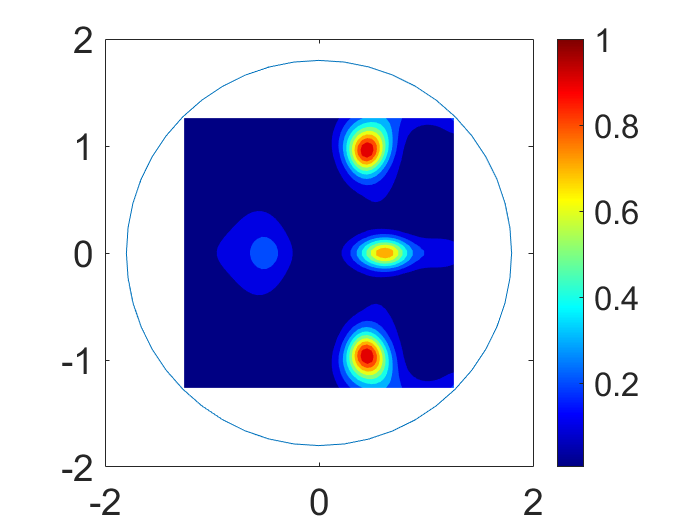}}
        \subfigure[DSM, $f_0=2$]{
            \includegraphics[width=0.3\textwidth]{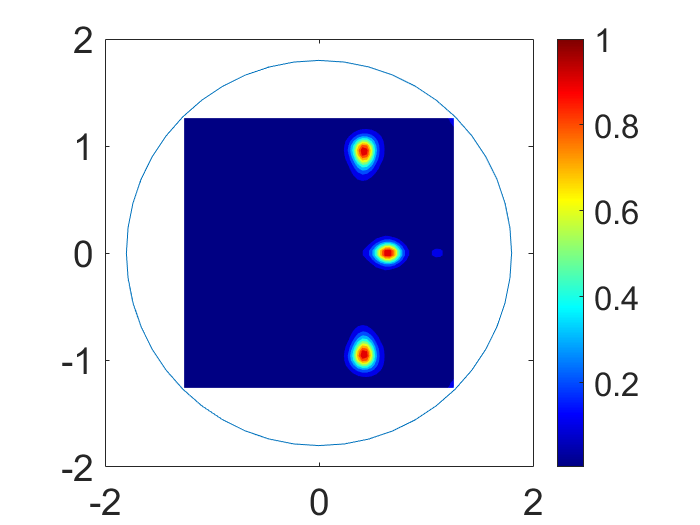}}
        \subfigure[DSM, $f_0=3$]{
            \includegraphics[width=0.3\textwidth]{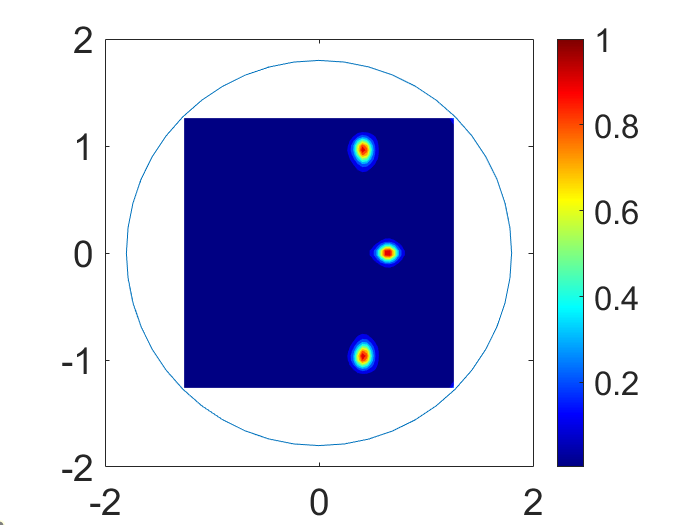}}
        \caption{\label{fig:frequency}Reconstruction of three small square scatterers using the total focusing method (TFM) and the time-domain direct sampling method (DSM) for different peak frequencies $f_0$.}
    \end{minipage}
\end{figure}

\subsection{Single Incident Source in Two Dimensions}\label{sec:Ricker}
In the two-dimensional setting, the measurement circle is uniformly discretized into $n$ observation points.
Furthermore, the indicator function \eqref{eq:time-indicator} is truncated at the terminal time $T$. The time interval $[0,T]$ is uniformly discretized into $m$ time steps with step size $\Delta t=T/m$, yielding the partition $t_0=0<t_1<\ldots<t_m=T$ and $t_k={kT}/{m}$. Thus, the indicator function is approximated by
\begin{equation}\label{eq:disIz}
\mathcal{I}(\bm{z})=\frac{T}{m}\sum_{i=1}^{m}\Big|\frac{2^{d-1}\pi r^{d-1}}{n}\sum_{j=1}^{n} \sum_{\tau=p,s} \frac{ e^{-\sigma(t_i+c_\tau^{-1}|\bm{x}^j-\bm{z}|)}}{|\bm{x}^j-\bm{z}|^{\frac{d-1}{2}}} \bm{\Gamma}_\tau(\bm{x}^j-\bm{z})\bm{u}^s(\bm{x}^j, t_i+c_\tau^{-1}|\bm{x}^j-\bm{z}|) \Big|^2,\quad \bm{z}\in \widetilde{D}.
\end{equation}

\begin{figure}[!ht]
\centering
\begin{minipage}{0.99\textwidth}
        \centering
\subfigure[SNR$=0$\,dB]{
\includegraphics[width=0.3\textwidth]{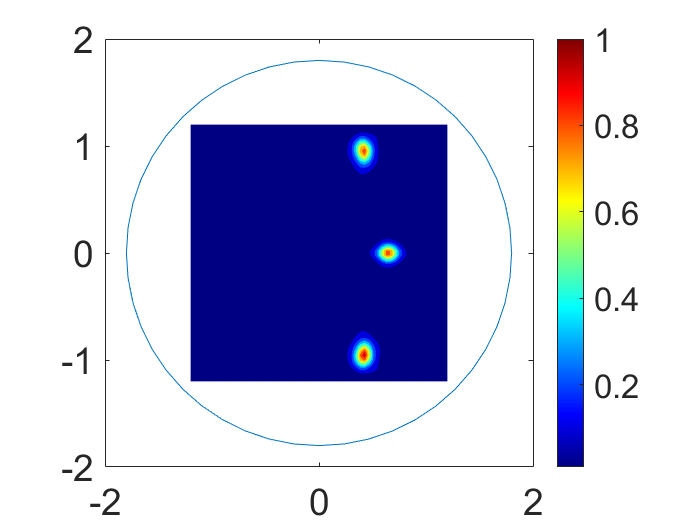}}
\subfigure[SNR$=-3$\,dB]{
\includegraphics[width=0.3\textwidth]{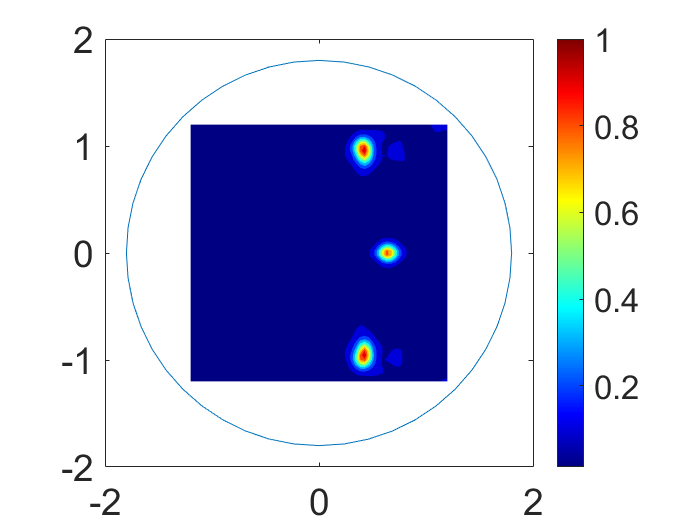}}
\subfigure[SNR$=-6$\,dB]{
\includegraphics[width=0.3\textwidth]{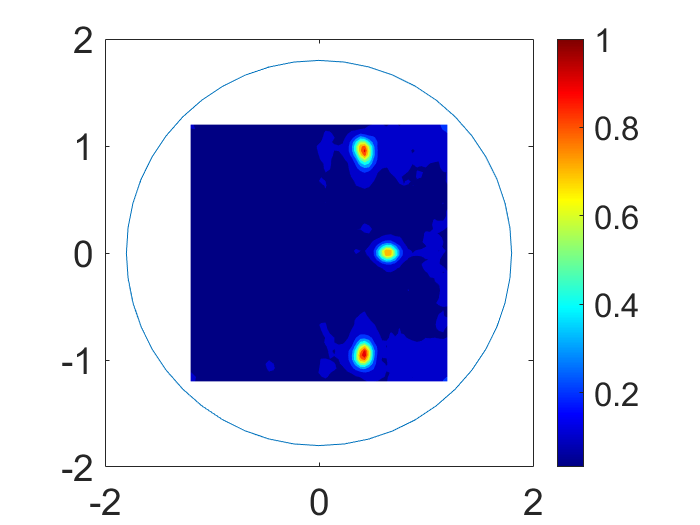}}
\subfigure[SNR$=-9$\,dB]{
\includegraphics[width=0.3\textwidth]{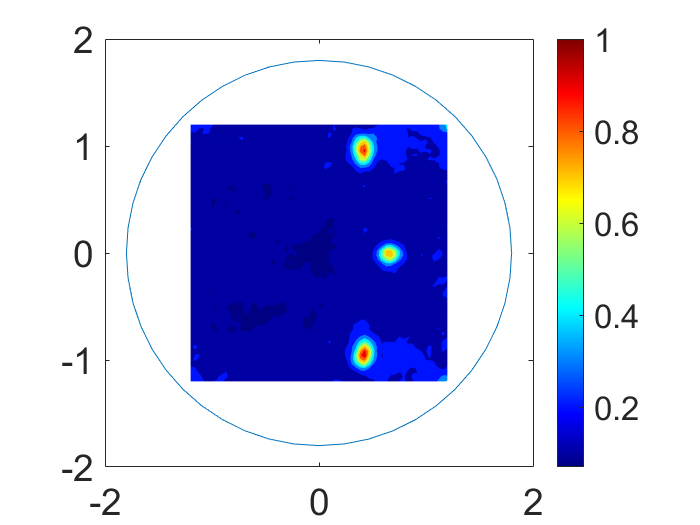}}
\subfigure[SNR$=-12$\,dB]{
\includegraphics[width=0.3\textwidth]{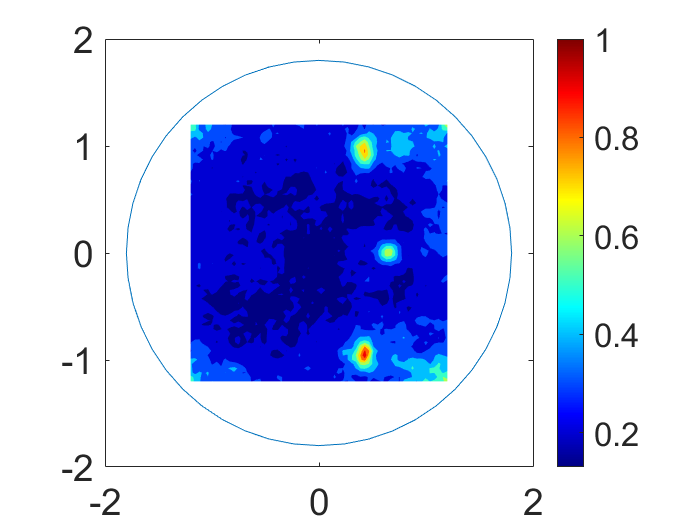}}
\subfigure[SNR$=-15$\,dB]{
\includegraphics[width=0.3\textwidth]{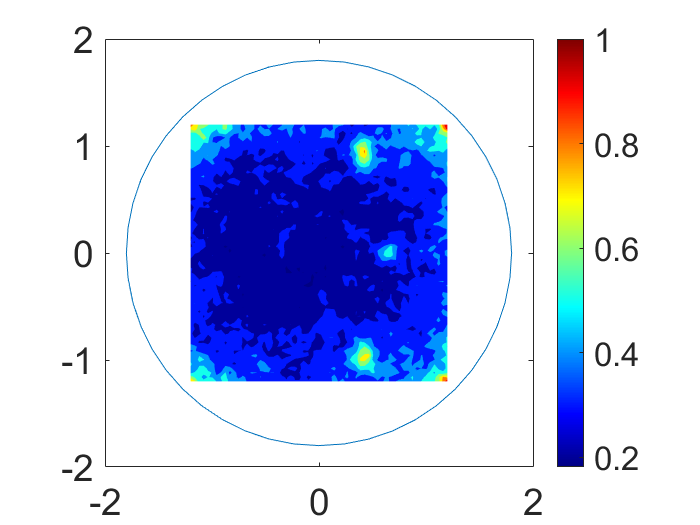}}
\caption{\label{fig:noise}Reconstruction of three small square scatterers under different noise levels $\delta$.}
    \end{minipage}
    \hfill
    \begin{minipage}{0.99\textwidth}
        \centering
\subfigure[$\theta\in(\frac{\pi}{4},\frac{7\pi}{4})$]{
\includegraphics[width=0.3\textwidth]{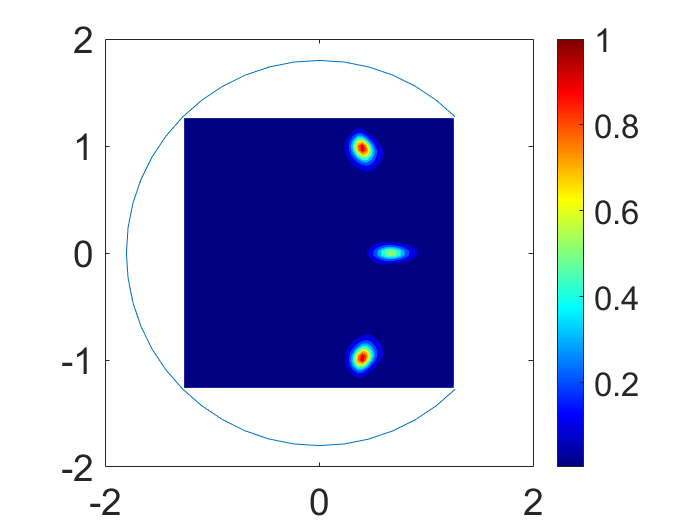}}
\subfigure[$\theta\in(\frac{\pi}{2},\frac{3\pi}{2})$]{
\includegraphics[width=0.3\textwidth]{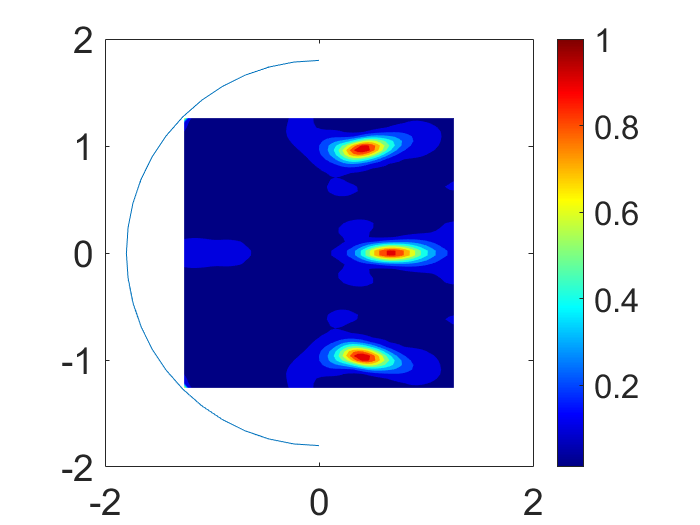}}
\subfigure[$\theta\in(\frac{3\pi}{4},\frac{5\pi}{4})$]{
\includegraphics[width=0.3\textwidth]{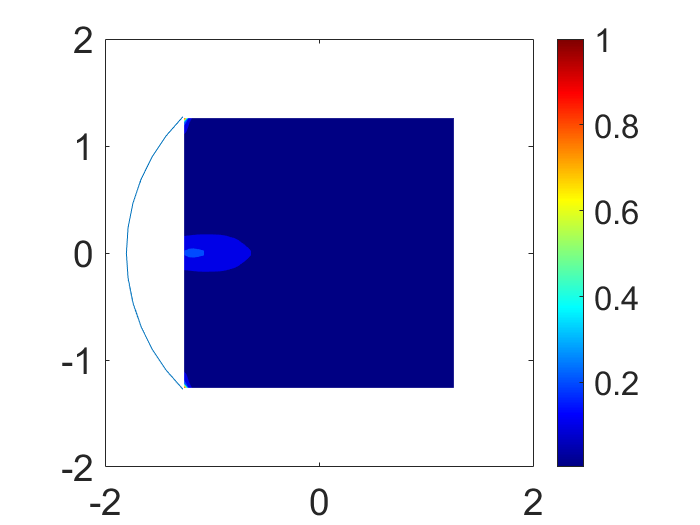}}\\
\subfigure[$\theta\in(-\frac{3\pi}{4},\frac{3\pi}{4})$]{
\includegraphics[width=0.3\textwidth]{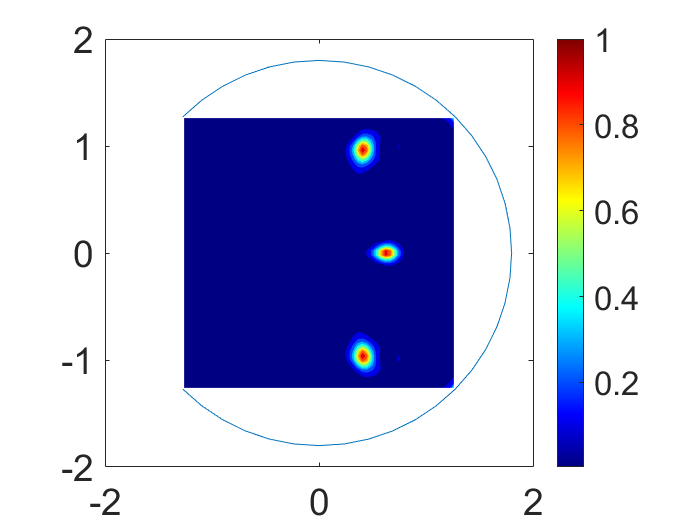}}
\subfigure[$\theta\in(-\frac{\pi}{2},\frac{\pi}{2})$]{
\includegraphics[width=0.3\textwidth]{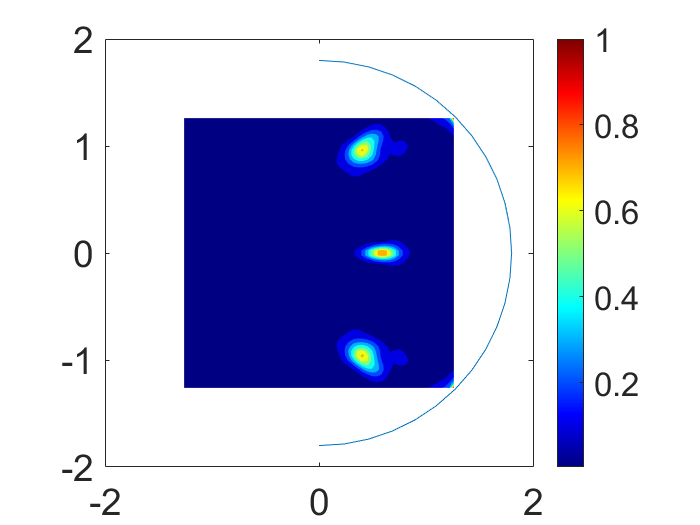}}
\subfigure[$\theta\in(-\frac{\pi}{4},\frac{\pi}{4})$]{
\includegraphics[width=0.3\textwidth]{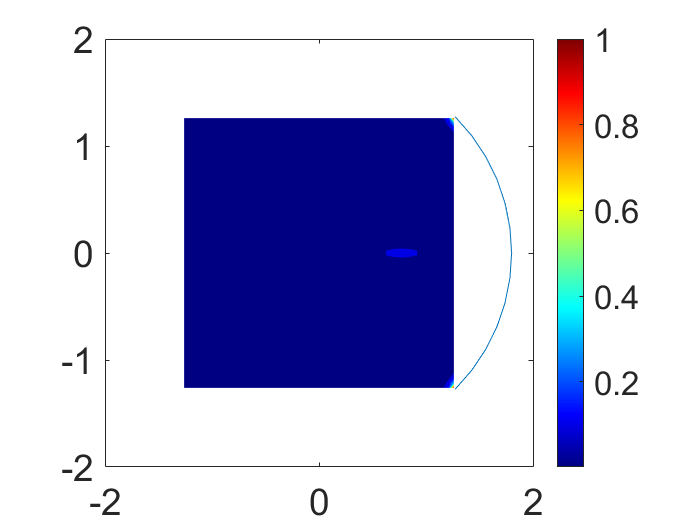}}
\caption{\label{fig:aperture}Reconstruction of three small square scatterers with limited-aperture data.}
    \end{minipage}
\end{figure}

For the first numerical example, the incident source is located at $(-1.5,0)$. Figure~\ref{fig:Ricker}(a) illustrates the computational configuration, where the red point denotes the incident source, the three small squares represent the scatterers, and the circle indicates the observation surface. The observation radius is set to $r=1.8$, and $n=48$ uniformly distributed observation points are used. The terminal time is set to $T=9$, and the time interval is uniformly discretized into $m=450$ time steps. The sampling domain is taken to be
$\widetilde{D}=[-1.26,1.26]\times[-1.26,1.26]$.
Furthermore, Figures~\ref{fig:Ricker}(b)--(c) show the Ricker wavelet and its corresponding Fourier spectra for different peak frequencies $f_0$, respectively.

To demonstrate the advantages of the proposed approach, we compare the time-domain direct sampling method with the total focusing method. The imaging functional of the total focusing method \cite{s25154550} is given by
\begin{equation*}
\mathcal{I}_{\mathrm{TFM}}(\bm{z})= \Big|\int_{ \Gamma} \bm{u}^s(\bm{x}, t_0+c_p^{-1}(|\bm{x}-\bm{z}|+|\bm{y}^0-\bm{z}|))\, \mathrm{d}\bm{x}\Big|,\quad \bm{z}\in \widetilde{D},
\end{equation*}
where $t_0$ is the time delay of the Ricker wavelet $\chi(t)$ and $\bm{y}^0$ is the location of the incident point source.
Figure~\ref{fig:frequency} presents the reconstruction results for three point-like scatterers obtained using the total focusing method (TFM) and the proposed time-domain direct sampling method (DSM) with incident waves of different peak frequencies $f_0$. Both imaging functionals attain local maxima at the scatterer locations, and the reconstruction quality improves as the peak frequency $f_0$ increases.
A comparison of the two methods shows that the proposed time-domain direct sampling method achieves higher imaging resolution and characterizes the scatterers more accurately than the total focusing method.

Next, we demonstrate the stability of the proposed method. Additive Gaussian noise corresponding to a prescribed signal-to-noise ratio (SNR) is added to the time-dependent measurements, i.e.,
\begin{equation*}
\bm u_{\delta}^s
=\bm u^s
+
\sqrt{\frac{P_{\mathrm{signal}}(\bm u^s)}{10^{\mathrm{SNR}/10}}}\, Z,
\end{equation*}
 where $P_{\mathrm{signal}}(\bm u^s)=\mathbb{E}(|\bm u^s|^2)$ denotes the signal power, $\mathbb{E}$ is the expectation operator over time $t$, and $Z\sim \mathcal{N}(0,1)$ is a standard Gaussian random variable.
Figure~\ref{fig:noise} shows the reconstruction results under various noise levels, with the signal-to-noise ratio ranging from $0$ dB to $-15$ dB. The proposed method achieves high-resolution reconstructions even under severe noise contamination, demonstrating its robustness with respect to measurement noise.
Furthermore, we investigate the performance of the proposed method under limited-aperture measurements. Figure~\ref{fig:aperture} presents the reconstructed images of three small square scatterers for observation surfaces with different apertures. The reconstruction quality gradually deteriorates as the aperture decreases.

\subsection{Multiple Incident Sources in Two Dimensions}

In this example, we consider the reconstruction of an extended scatterer. The scatterer has a kite-shaped geometry, as shown in Figure~\ref{subfig:kite}, and is parameterized by
\begin{equation*}
\bm{x}(\phi)=(0.4\cos(\phi)  + 0.26\cos(2\phi),0.6\sin(\phi)),\quad \phi\in[0,2\pi].
\end{equation*}
We first evaluate the proposed imaging functional \eqref{eq:disIz} using the scattered field generated by a single incident source. As shown in Figure~\ref{subfig:single}, only the portion of the scatterer closest to the source is reconstructed. To overcome this limitation, we introduce an imaging functional that incorporates multiple incident sources:
\begin{equation}\label{eq:disIz2}
\begin{aligned}
\widetilde{\mathcal{I}}(\bm{z})=\int_{-\infty}^{\infty}\int_{\widetilde{\Gamma}} \Big|\int_{ \Gamma} \sum_{\tau=p,s} \frac{ e^{-\sigma(t+c_\tau^{-1}|\bm{x}-\bm{z}|)}}{\sqrt{|\bm{x}-\bm{z}|}} \bm{\Gamma}_\tau(\bm{x}-\bm{z})\bm{u}^s(\bm{x}, \tilde{\bm{x}},t+c_\tau^{-1}|\bm{x}-\bm{z}|)  \,\mbox{d}s(\bm{x})\Big|^2\mathrm{d}\tilde{\bm{x}}\,\mbox{d}t,\quad \bm{z}\in \widetilde{D},
\end{aligned}
\end{equation}
where $\widetilde{\Gamma}$ denotes the surface on which the incident sources are placed. We then discretize $\widetilde{\Gamma}$ into a uniformly distributed array of incident sources. The discretization of the observation surface $\Gamma$ and the time interval follows the same strategy as that used for the indicator function \eqref{eq:disIz}.

The reconstructions generated by the discretized indicator $\widetilde{\mathcal{I}}(\bm{z})$ are presented in Figures~\ref{subfig:8}--\ref{subfig:32}. The source radius is set to $\tilde{r}=1.5$, and the number of incident sources is set to $\tilde{n}=8,16,24,32$, respectively. All other parameters are chosen as in Section~\ref{sec:Ricker}. The black points arranged in a circle indicate the positions of the incident sources, while those enclosed by a small red circle are the active sources for each illumination. As shown in Figures~\ref{subfig:8}--\ref{subfig:32}, the reconstruction quality improves as the number of incident sources increases, resulting in a more complete recovery of the scatterer. Furthermore, the portions of the scatterer facing the incident sources are reconstructed more accurately, which is consistent with the observation in Figure~\ref{subfig:single}.

\FloatBarrier

\begin{figure}[H]
\centering
\subfigure[\label{subfig:kite} computational geometry]{
\includegraphics[width=0.3\textwidth]{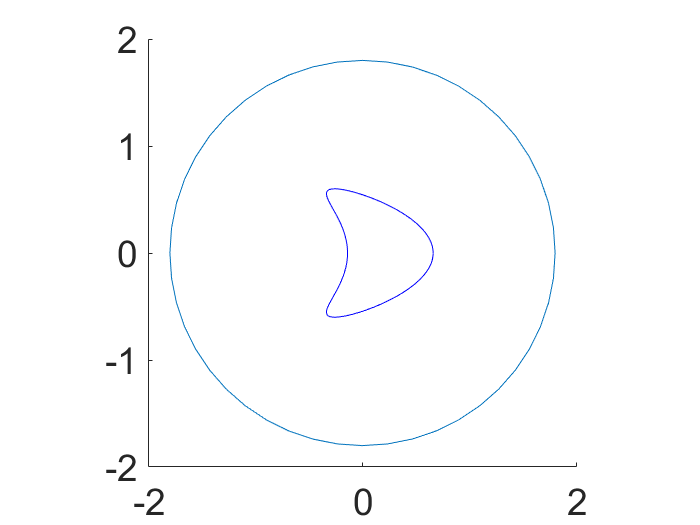}}
\subfigure[\label{subfig:single} single incident source]{
\includegraphics[width=0.3\textwidth]{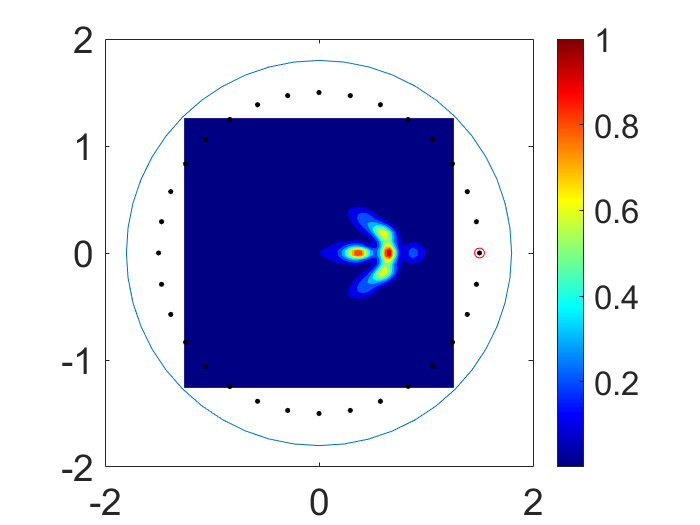}}
\subfigure[\label{subfig:8} 8 incident sources]{
\includegraphics[width=0.3\textwidth]{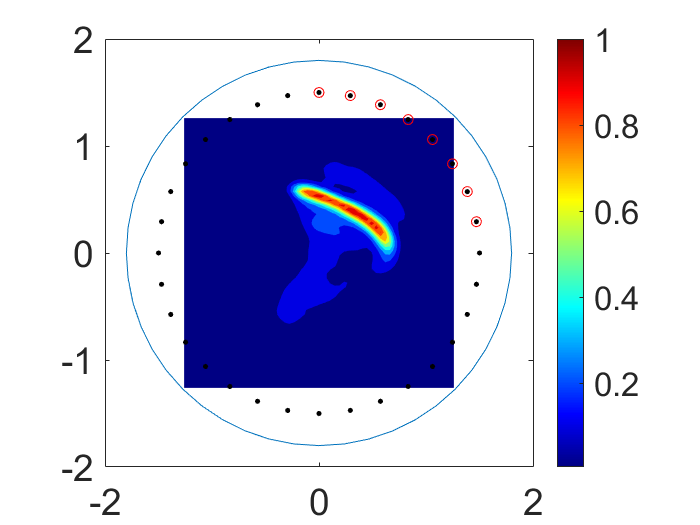}}\\
\subfigure[\label{subfig:16} 16 incident sources]{
\includegraphics[width=0.3\textwidth]{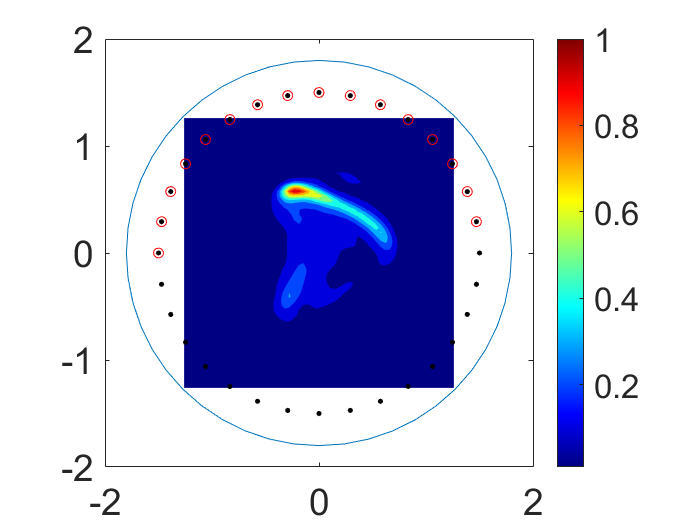}}
\subfigure[\label{subfig:24} 24 incident sources]{
\includegraphics[width=0.3\textwidth]{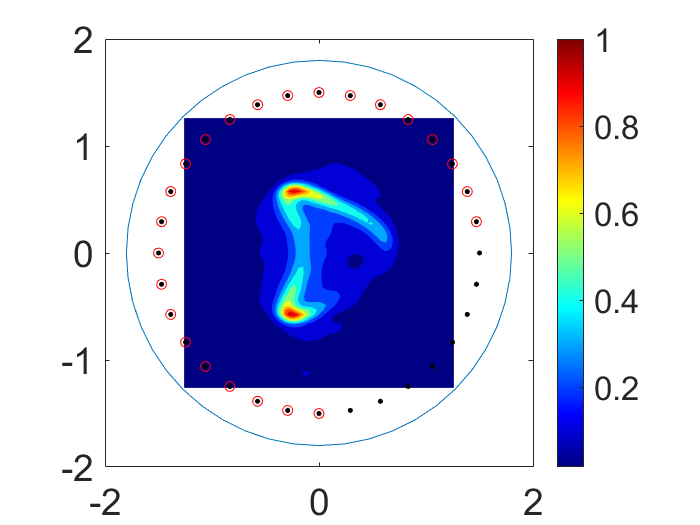}}
\subfigure[\label{subfig:32} 32 incident sources]{
\includegraphics[width=0.3\textwidth]{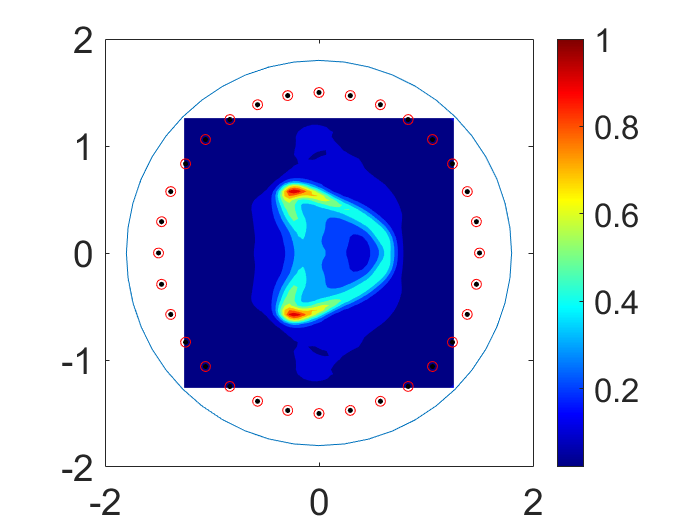}}
\caption{\label{fig:re-kite}Reconstruction of a kite-shaped scatterer using a single incident source and multiple incident sources.}
\end{figure}

\subsection{Single Incident Source in Three Dimensions}
In the final example, we investigate the identification of two small scatterers in three-dimensional space using a single incident source. The point source is located at $(-1.5,0,0)$, as indicated by the red marker in Figure~\ref{subfig:geo3D}. The two small cubes represent the scatterers under investigation. The spherical surface denotes the observation domain, and the discretized points on it correspond to the measurement locations. For improved visualization, projections of the scatterers onto the planes $z_1=1$, $z_2=1$, and $z_3=-1$ are also presented in Figure~\ref{subfig:geo3Dproj}.

The reconstruction results are shown in Figures~\ref{subfig:slice3D} and \ref{subfig:iso3D}. The sampling domain $\widetilde{D}$ is taken as the cube $[-1,1]\times[-1,1]\times[-1,1]$. In Figure~\ref{subfig:slice3D}, the indicator function $\mathcal{I}(\bm{z})$ is visualized on the slices $z_1=0,z_2=\pm0.5$ and $z_3=\pm0.5$, which clearly reveal the locations of its peak values. Furthermore, Figure~\ref{subfig:iso3D} displays the isosurfaces of $\mathcal{I}(\bm{z})$ together with its projections onto the planes $z_1=1,z_2=1$, and $z_3=-1$. A comparison between the true scatterers in Figure~\ref{fig:re3D}(b) and the corresponding reconstruction in Figure~\ref{fig:re3D}(d) demonstrates that the proposed direct sampling method also yields satisfactory reconstructions in the three-dimensional setting.

\begin{figure}[!ht]
\centering
\subfigure[computational geometry\label{subfig:geo3D}]{
\includegraphics[width=0.45\textwidth]{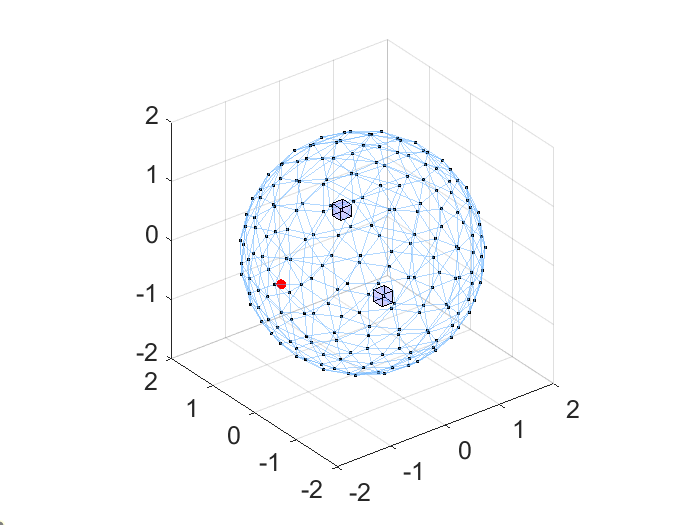}}
\subfigure[exact scatterers\label{subfig:geo3Dproj}]{
\includegraphics[width=0.45\textwidth]{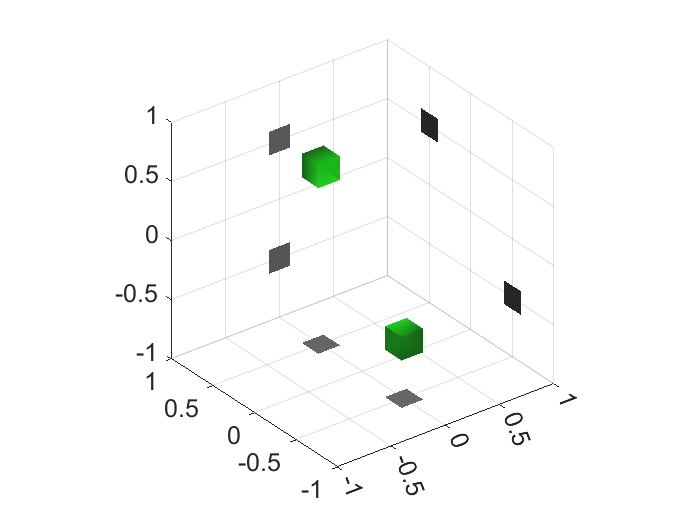}}\\
\subfigure[slice plots of the reconstruction\label{subfig:slice3D}]{
\includegraphics[width=0.45\textwidth]{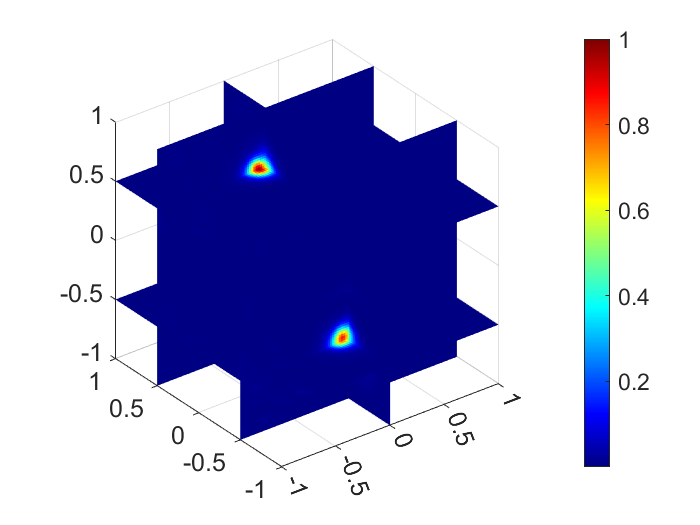}}
\subfigure[isosurface plot of the reconstruction\label{subfig:iso3D}]{
\includegraphics[width=0.45\textwidth]{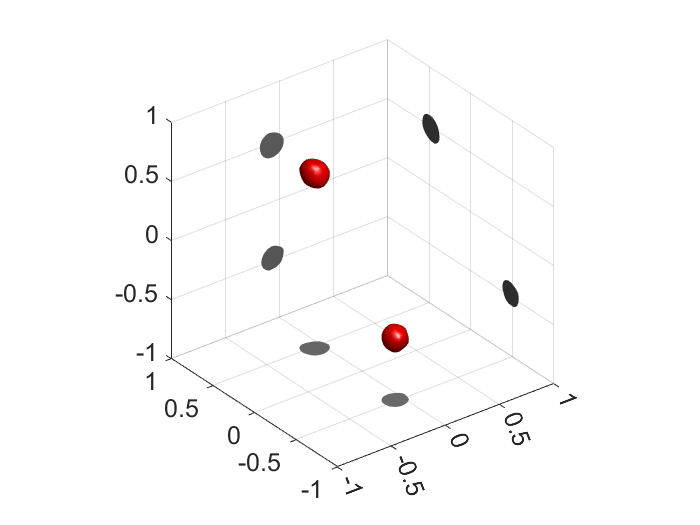}}
\caption{\label{fig:re3D}Plots of the exact and reconstructed scatterers in three dimensions.}
\end{figure}

\appendix

\section{Direct Sampling Method in the Frequency Domain}
In this appendix, we present a non-decoupled direct sampling method in the frequency domain that is applicable to both two- and three-dimensional elastic wave problems. This approach provides an efficient tool for locating scatterers from full-aperture measurement data without requiring iterative inversion procedures.

\subsection{Mathematical Formulation}
We consider the same configuration as that described in \eqref{eq:confi}. Let $\hat{\bm{u}}^i$ denote the time-harmonic incident elastic wave satisfying the governing elastic equation in the whole space $\mathbb{R}^d$, $d=2,3$:
\begin{equation}\label{eq:inci}
\mathcal{L}_{ {\lambda_2}, {\mu_2}}\hat{\bu}^i(\bx) + \omega^2  {\rho_2} \hat{\bu}^i(\bx) =0, \quad \bx\in\mathbb{R}^d.
\end{equation}
The total displacement field $\hat{\bm{u}}$ satisfies
\begin{equation}\label{eq:to1}
\mathcal{L}_{ {\lambda}, {\mu}}\hat{\bu}(\bx)+\omega^2 \rho(\bm{x})\hat{\bu}(\bx)=0,\quad \bx\in\mathbb{R}^d,
\end{equation}
and the scattered wave $\hat{\bu}^s:=\hat{\bu}-\hat{\bu}^i$ satisfies the radiation condition
\begin{equation*}
\begin{aligned}
(\nabla\times\nabla\times \hat{\bu}^s)(\bx)\times\frac{\bx}{|\bx|}-\mathrm{i} {k}_{s}\nabla\times \hat{\bu}^s(\bx)=&\mathcal{O}(|\bx|^{-2}),\\
\frac{\bx}{|\bx|}\cdot[\nabla(\nabla\cdot \hat{\bu}^s)](\bx)-\mathrm{i} {k}_{p}\nabla \hat{\bu}^s(\bx)=&\mathcal{O}(|\bx|^{-2}),
\end{aligned}
\end{equation*}
as $\|\bx\|\rightarrow \infty$, where $k_s=\omega/c_s$ and $k_p=\omega/c_p$, with $c_s$ and $c_p$ defined in Section~\ref{sec:1}. Here, $k_s$ and $k_p$ denote the wavenumbers corresponding to the shear and pressure waves of the background medium.

To facilitate the analysis, we introduce the perturbation parameters
\begin{equation*}
\lambda_3(\bx)=\lambda(\bx)-\lambda_2 , \quad \mu_3(\bx)=\mu(\bx)-\mu_2, \quad \rho_3(\bx)=\rho(\bx) -\rho_2,
\end{equation*}
which are compactly supported in the scatterer domain $D$. Then \eqref{eq:to1} can be rewritten as
\begin{equation}\label{eq:reto1}
\mathcal{L}_{ {\lambda_2}, {\mu_2}}\hat{\bu}(\bx) + \omega^2  {\rho_2} \hat{\bu}(\bx) = -\mathcal{L}_{ {\lambda_3}, {\mu_3}}\hat{\bu}(\bx) - \omega^2\rho_3(\bx)\hat{\bu}(\bx).
\end{equation}
Subtracting \eqref{eq:inci} from \eqref{eq:reto1} gives
\begin{equation*}
\mathcal{L}_{ {\lambda_2}, {\mu_2}}\hat{\bu}^s(\bx) + \omega^2  {\rho_2} \hat{\bu}^s(\bx) = -\mathcal{L}_{ {\lambda_3}, {\mu_3}}\hat{\bu}(\bx) - \omega^2\rho_3(\bx)\hat{\bu}(\bx).
\end{equation*}
Defining
\begin{equation*}
\bm{f}(\bx)=\mathcal{L}_{ {\lambda_3}, {\mu_3}}\hat{\bu}(\bx) + \omega^2\rho_3(\bx)\hat{\bu}(\bx),
\end{equation*}
the Lippmann--Schwinger representation yields
\begin{equation}\label{eq:fresol}
\hat{\bu}^s(\bx)=\int_D \bGa^\omega(\bx,\by)\bm{f}(\by) \,\mathrm{d}\by,
\end{equation}
where $\bGa^{\omega}=(\Gamma^{\omega}_{i,j})_{i,j=1}^d$ is the fundamental solution of the background elastic operator; see \eqref{eq:Gamma}. In the static case $\omega=0$, the corresponding fundamental solution $\bm{\Gamma}^0$ is
\begin{equation*}
    \Gamma^0_{i,j} (\bx)= \left\{
    \begin{aligned}
    &-\frac{1}{4\pi}\Big( \frac{1}{\mu} + \frac{1}{\lambda + 2\mu} \Big)  \delta_{ij} \ln|\bx|  +\frac{1}{4\pi}\Big( \frac{1}{\mu} - \frac{1}{\lambda + 2\mu} \Big)  \frac{\bx_i \bx_j}{|\bx|^2 }, &d=2, \\
    &-\frac{1}{8\pi}\Big( \frac{1}{\mu} + \frac{1}{\lambda + 2\mu} \Big)  \frac{\delta_{ij}}{|\bx|}  -\frac{1}{8\pi}\Big( \frac{1}{\mu} - \frac{1}{\lambda + 2\mu} \Big)  \frac{\bx_i \bx_j}{|\bx|^3 }, &d=3.
    \end{aligned}\right.
   \end{equation*}

\subsection{Frequency-Domain Direct Sampling Method}
Assume that full-aperture measurement data are collected on an observation surface $\Gamma\subset \mathbb{R}^d\backslash \overline{D}$, which, for simplicity, is taken to be a sphere of radius $r$. The frequency-domain indicator function is defined by
\begin{equation*}
\hat{\mathcal{I}}(\bz)=\int_\Gamma \overline{\Upsilon^\omega (\bx,\bz)}\hat{\bu}^s(\bx) \,\mathrm{d}s(\bx),
\end{equation*}
where $\Upsilon^\omega (\bx,\by)$ is given by
\begin{equation*}
\Upsilon^{\omega}(\bx,\by) =   b_{d,1} \bGa_p(\bx) \frac{e^{\rmi k_p (|\bm{x}|- \hat{\bx}\cdot\by )}}{|\bx|^{\frac{d-1}{2}}} + b_{d,2} \bGa_s(\bx) \frac{e^{\rmi k_s (|\bm{x}|- \hat{\bx}\cdot\by )}}{|\bm{x}|^{\frac{d-1}{2}}},
\end{equation*}
and the parameters $b_{d,1}$ and $b_{d,2}$ are given by
\[b_{2,1}=(\lambda_2+2\mu_2)^{3/2}, \quad b_{2,2}=\mu_2^{3/2},\quad b_{3,1}=(\lambda_2+2\mu_2)^2, \quad b_{3,2}=\mu_2^2.\]
This indicator function is designed to exhibit local maxima near the scatterer locations, thereby providing a non-iterative approach for detecting the presence and approximate positions of inclusions.

To analyze the behavior of $\hat{\mathcal{I}}(\bz)$, we introduce
\begin{equation}\label{eq:defG}
\mathbf{G}(\bz,\by)=\int_\Gamma \overline{\Upsilon^\omega (\bx,\bz)} \bGa^\omega(\bx,\by)\,\mathrm{d}s(\bx),
\end{equation}
which captures the interaction between the sampling point $\bz$ and the source point $\by$.

\begin{lem}\label{lem:bessel}
For $\bm{\alpha}\in \mathbb{R}^3$, the following identities hold:

\begin{equation*}
    \begin{aligned}
    &\int_{ \mathbb{S}^2} \mathbf{I} e^{-\mathrm{i} \hat{\bm{x}}\cdot \bm{\alpha}}\,\mathrm{d}s(\hat{\bm{x}}) = 4\pi j_0( |\bm{\alpha}|)\mathbf{I},\\
    &\int_{ \mathbb{S}^2} \hat{\bm{x}}\hat{\bm{x}}^\top e^{-\mathrm{i}\hat{\bm{x}}\cdot \bm{\alpha}}\,\mathrm{d}s(\hat{\bm{x}}) = \frac{4\pi}{3}(j_0(c^{-1}|\bm{\alpha}|)+j_2(c^{-1}|\bm{\alpha}|))\mathbf{I} -4\pi j_2( c^{-1}|\bm{\alpha}|)\hat{\bm{\alpha}}\hat{\bm{\alpha}}^T.
    \end{aligned}
    \end{equation*}
    Furthermore, in the two-dimensional case, for $\bm{\alpha}\in \mathbb{R}^2$, one has
    \begin{equation*}
    \begin{aligned}
    &\int_{ \mathbb{S}^1} \mathbf{I} e^{-\mathrm{i} \hat{\bm{x}}\cdot \bm{\alpha}}\,\mathrm{d}s(\hat{\bm{x}}) = 2\pi J_0( |\bm{\alpha}|)\mathbf{I},\\
    &\int_{ \mathbb{S}^1} \hat{\bm{x}}\hat{\bm{x}}^\top e^{-\mathrm{i}\hat{\bm{x}}\cdot \bm{\alpha}}\,\mathrm{d}s(\hat{\bm{x}}) = \pi(J_0(|\bm{\alpha}|)+J_2(|\bm{\alpha}|)) \mathbf{I} -2\pi J_0(|\bm{\alpha}|)\hat{\bm{\alpha}}\hat{\bm{\alpha}}^T.
    \end{aligned}
    \end{equation*}
    Here, $j_n$ denotes the spherical Bessel function of the first kind of order $n$, and $J_n$ denotes the cylindrical Bessel function of the first kind of order $n$.
\end{lem}
\begin{proof}
Applying the Funk--Hecke formula \cite{Colton_2019}
\begin{equation*}
    \int_{ \mathbb{S}^2}e^{-\mathrm{i}\hat{\bm{x}}\cdot \bm{\alpha}} Y_n^m(\hat{\bm{x}})\,\mbox{d}s(\hat{\bm{x}})=4\pi  (-\mathrm{i})^n  j_n( |\bm{\alpha}|)Y_n^m(\hat{\bm{\alpha}}),
\end{equation*}
together with the explicit forms of the spherical harmonic functions in \eqref{eq:harmonic} yields the three-dimensional identities. Similarly, the two-dimensional case follows from the Jacobi--Anger expansion
    \begin{equation*}
    e^{-\mathrm{i}\hat{\bm{x}}\cdot\bm{\alpha}}=J_0(|\bm{\alpha}|)+ 2\sum_{n=1}^{\infty}(-\mathrm{i})^n J_n(|\bm{\alpha}|) \cos(n\varphi),
    \end{equation*}
    where $\varphi$ is the angle between $\bm{x}$ and $\bm{\alpha}$.
\end{proof}

\begin{lem}\label{lem:G_decay}
The function $\mathbf{G}(\bz,\by)$ defined in \eqref{eq:defG} decays to zero as $\bz$ moves away from $\by$. More precisely,
\begin{equation*}
\mathbf{G}(\bz,\by)=\mathcal{O}\Big(\frac{1}{|\bz-\by|^{\frac{d-1}{2}}}\Big) + \mathcal{O}\Big(\frac{1}{r^{\frac{d-1}{2}}}\Big),\quad |\bz-\by|\to\infty,\, r\to\infty.
\end{equation*}
\end{lem}

\begin{proof}
    By a direct calculation using Lemma \ref{lem:gamma_ps}, we have
    \begin{equation}\label{eq:freqG}
    \mathbf{G}(\bz,\by)=-\frac{a_{d,1} b_{d,1}}{\omega^{\frac{3-d}{2}}} \mathbf{G}_p(\bz,\by)-\frac{a_{d,2} b_{d,2}}{\omega^{\frac{3-d}{2}}}\mathbf{G}_s(\bz,\by) +\mathcal{O}\Big(\frac{1}{r^{\frac{d-1}{2}}}\Big),
    \end{equation}
    where
    \begin{align*}
    \mathbf{G}_p(\bz,\by)=\int_{\mathbb{S}^{d-1}} \bGa_p(\bx) e^{-\mathrm{i}k_p\hat{\bm{x}}\cdot(\bz-\by)}\,\mathrm{d}s(\bx),\quad
    \mathbf{G}_s(\bz,\by)=\int_{\mathbb{S}^{d-1}} \bGa_p(\bx) e^{-\mathrm{i}k_s\hat{\bm{x}}\cdot(\bz-\by)}\,\mathrm{d}s(\bx).
    \end{align*}
    Let $\bm{\eta}=\bz-\by$. Using Lemma \ref{lem:bessel}, we obtain
    \begin{align*}
    \mathbf{G}_p(\bz,\by)&=\left\{
    \begin{aligned}
    -\pi(J_2(k_p|\bm{\eta}|)+J_0(k_p|\bm{\eta}|))\mathbf{I} +2\pi J_2( k_p|\bm{\eta}|)\hat{\bm{\eta}}\hat{\bm{\eta}}^\top,\quad &d=2,\\
    -\frac{4\pi}{3}(j_2(k_p|\bm{\eta}|)+j_0(k_p|\bm{\eta}|))\mathbf{I} +4\pi j_2( k_p|\bm{\eta}|)\hat{\bm{\eta}}\hat{\bm{\eta}}^\top ,\quad &d=3,
    \end{aligned}\right.\\
    \mathbf{G}_s(\bz,\by)&=\left\{
    \begin{aligned}
    \pi(J_2(k_s|\bm{\eta}|)-J_0(k_s|\bm{\eta}|))\mathbf{I} -2\pi J_2( k_s|\bm{\eta}|)\hat{\bm{\eta}}\hat{\bm{\eta}}^\top,\quad &d=2,\\
    \frac{4\pi}{3}(j_2(k_s|\bm{\eta}|)-2j_0(k_s|\bm{\eta}|))\mathbf{I} -4\pi j_2( k_s|\bm{\eta}|)\hat{\bm{\eta}}\hat{\bm{\eta}}^\top,\quad &d=3.
    \end{aligned}\right.
    \end{align*}
    For large arguments, the Bessel function $J_n(s)$ and the spherical Bessel function $j_n(s)$, $n=0,2$, admit the asymptotic expansions
    \begin{equation*}
    \begin{aligned}
    J_n(s)&=\sqrt{\frac{2}{\pi s}} \cos\big(s-\frac{n\pi}{2}-\frac{\pi}{4}\big) \big\{1+\mathcal{O}\big(\frac{1}{s}\big)\big\},\\
    j_n(s)&=\frac{1}{ s} \cos\big(s-\frac{n\pi}{2}-\frac{\pi}{2}\big) \big\{1+\mathcal{O}\big(\frac{1}{s}\big)\big\},
    \end{aligned}
    \end{equation*}
    as $s\rightarrow\infty$. Hence, for $d=2$, the terms involving $J_0$ and $J_2$ are of order $\mathcal{O}\big(\frac{1}{|\bz-\by|^{1/2}}\big)$, whereas for $d=3$, the terms involving $j_0$ and $j_2$ are of order $\mathcal{O}\big(\frac{1}{|\bz-\by|}\big)$. Therefore, in both two and three dimensions,
    \begin{equation*}
    \mathbf{G}_p(\bz,\by)=\mathcal{O}\Big(\frac{1}{|\bz-\by|^{\frac{d-1}{2}}}\Big),
    \quad
    \mathbf{G}_s(\bz,\by)=\mathcal{O}\Big(\frac{1}{|\bz-\by|^{\frac{d-1}{2}}}\Big),
    \end{equation*}
    as $|\bz-\by|\rightarrow\infty$. Substituting these estimates into the representation formula for $\mathbf{G}(\bz,\by)$ yields
    \begin{equation*}
    \mathbf{G}(\bz,\by)=\mathcal{O}\Big(\frac{1}{|\bz-\by|^{\frac{d-1}{2}}}\Big)
    +\mathcal{O}\Big(\frac{1}{r^{\frac{d-1}{2}}}\Big).
    \end{equation*}
    This completes the proof.
    \end{proof}

\begin{lem}\label{lem:G}
Let the function $\mathbf{G}(\bz,\by)$ be defined by \eqref{eq:defG}. For $\bz=\by$, one has
\begin{equation*}
\mathbf{G}(\bz,\by)=
\left\{
\begin{aligned}
&\frac{\sqrt{\pi}(1+\mathrm{i}) (\lambda_2+3\mu_2)}{4\sqrt{\omega}}\mathbf{I} + \mathcal{O}\Big(\frac{1}{\sqrt{r}}\Big), & d=2,\\
&\frac{\lambda_2+4\mu_2}{3}\mathbf{I} + \mathcal{O}\Big(\frac{1}{r}\Big), & d=3,
\end{aligned}\right.
\end{equation*}
and for $\bz$ near $\by$,
\begin{equation*}
\mathbf{G}(\bz,\by)=
\left\{
\begin{aligned}
&\frac{\sqrt{\pi}(1+\mathrm{i})}{4\sqrt{\omega}}\big(\lambda_2+3\mu_2 - \frac{3\omega^2 \rho_2}{4}|\bz-\by|^2\big)\mathbf{I} + \mathcal{O}(|\bz-\by|^4), & d=2,\\
&\big(\frac{\lambda_2+4\mu_2}{3} - \frac{\omega^2 \rho_2}{6}|\bz-\by|^2\big)\mathbf{I} + \mathcal{O}(|\bz-\by|^4), & d=3.
\end{aligned}\right.
\end{equation*}
\end{lem}
\begin{proof}
    We note that the Bessel function $J_n(s)$ and the spherical Bessel function $j_n(s)$, $n=0,2$, admit the asymptotic expansions
    \begin{equation*}
    \begin{aligned}
    J_0(s)=1-\frac{s^2}{4}+\mathcal{O}(s^4),\quad J_2(s)=\frac{s^2}{8}+\mathcal{O}(s^4),\\
    j_0(s)=1-\frac{s^2}{6}+\mathcal{O}(s^4),\quad j_2(s)=\frac{s^2}{15}+\mathcal{O}(s^4),\\
    \end{aligned}
    \end{equation*}
    as $s\rightarrow 0$. Substituting these expansions into \eqref{eq:freqG} and simplifying the resulting expression yield the desired result.
    \end{proof}

We are now ready to state the properties of the indicator function $\hat{\mathcal{I}}(\bz)$.
\begin{thm}
 Assume that the scatterer configuration is given by \eqref{eq:Dscatters}. Then the indicator function satisfies the following properties:
\begin{itemize}
\item If $\bz$ lies near a scatterer $D_j$, then, as $\varepsilon\rightarrow0$, $|\bz-\by^j|\rightarrow 0$, $r\rightarrow\infty$, and $L\rightarrow\infty$,
\begin{equation*}
\hat{\mathcal{I}}(\bz) = \varepsilon^d C_d \int_{B_j} \bm{f}(\by^j+\varepsilon\bm{\zeta})\,\mathrm{d}\bm{\zeta} 
\Big\{
1 + \mathcal{O}(|\bz-\by^j|^2) + \mathcal{O}\big(\frac{1}{r^{\frac{d-1}{2}}}\big) + \mathcal{O}(\varepsilon) + \mathcal{O}\big(\frac{1}{L^{\frac{d-1}{2}}}\big)
\Big\},
\end{equation*}
where
    \begin{equation*}
    C_2=\frac{\sqrt{\pi}(1+\mathrm{i}) (\lambda_2+3\mu_2)}{4\sqrt{\omega}},\quad C_3= \frac{\lambda_2+4\mu_2}{3}.
    \end{equation*}
\item If $\bz$ is far from all scatterers, then, as $\varepsilon\rightarrow0$ and $r\rightarrow\infty$,
\begin{equation*}
\hat{\mathcal{I}}(\bz) = \varepsilon^d \sum_{j=1}^N \int_{B_j} \bm{f}(\by^j+\varepsilon\bm{\zeta})\,\mathrm{d}\bm{\zeta} 
\Big\{
\mathcal{O}\big(\frac{1}{r^{\frac{d-1}{2}}}\big) + \mathcal{O}(\varepsilon) + \mathcal{O}\big(\frac{1}{\mathrm{dist}(\bz,D)^{\frac{d-1}{2}}}\big)
\Big\}.
\end{equation*}
\end{itemize}
\end{thm}
 
\begin{proof}
    Applying the Taylor expansion to \eqref{eq:fresol}, the scattered wave $\hat{\bu}^s(\bx)$ can be written as
    \begin{equation*}
    \hat{\bu}^s(\bx)=\varepsilon^d \sum_{j=1}^{N} \bGa^\omega(\bx,\by^j) \Big(\int_{B_j} \bm{f}(\by^j+\varepsilon\bm{\zeta})\,\mathrm{d}\bm{\zeta} +\mathcal{O}(\varepsilon)\Big).
    \end{equation*}
    Consequently,
    \begin{equation*}
    \begin{aligned}
    \hat{\mathcal{I}}(\bz)=&\int_\Gamma \overline{\Upsilon^\omega (\bx,\bz)}\hat{\bu}^s(\bx) \,\mathrm{d}s(\bx)\\
    =&\varepsilon^d \sum_{j=1}^{N} \int_\Gamma \overline{\Upsilon^\omega (\bx,\bz)} \bGa^\omega(\bx,\by^j) \Big(\int_{B_j} \bm{f}(\by^j+\varepsilon\bm{\zeta})\,\mathrm{d}\bm{\zeta} +\mathcal{O}(\varepsilon)\Big) \,\mathrm{d}s(\bx)\\
    =&\varepsilon^d \sum_{j=1}^{N}\mathbf{G}(\bz,\by^j)\Big( \int_{B_j}  \bm{f}(\by^j+\varepsilon\bm{\zeta})\,\mathrm{d}\bm{\zeta} +\mathcal{O}(\varepsilon)\Big).
    \end{aligned}
    \end{equation*}
    Using the properties of $\mathbf{G}(\bz,\by)$ stated in Lemmas \ref{lem:G_decay} and \ref{lem:G}, if the sampling point $\bm{z}$ lies in a neighborhood of the scatterer $D_j$, then
    \begin{equation*}
    \begin{aligned}
    \hat{\mathcal{I}}(\bz)=&\varepsilon^d \mathbf{G}(\bz,\by^j)\Big( \int_{B_j}  \bm{f}(\by^j+\varepsilon\bm{\zeta})\,\mathrm{d}\bm{\zeta} +\mathcal{O}(\varepsilon)\Big)+\varepsilon^d\sum_{\substack{k=1\\ k\neq j}}^N \mathbf{G}(\bz,\by^k)\Big( \int_{B_k}  \bm{f}(\by^k+\varepsilon\bm{\zeta})\,\mathrm{d}\bm{\zeta} +\mathcal{O}(\varepsilon)\Big)\\
    =&\varepsilon^d C_d \int_{B_j} \bm{f}(\by^j+\varepsilon\bm{\zeta})\,\mathrm{d}\bm{\zeta}\Big\{1 +\mathcal{O}(|\bz-\by^j|^2) +\mathcal{O}\Big(\frac{1}{r^{\frac{d-1}{2}}}\Big) +\mathcal{O}(\varepsilon) +\mathcal{O}\Big(\frac{1}{L^{\frac{d-1}{2}}}\Big) \Big\},
    \end{aligned}
    \end{equation*}
    as $\varepsilon\rightarrow0$, $|\bz-\by^j|\rightarrow 0$, $r\rightarrow\infty$, and $L\rightarrow\infty$. On the other hand, if the sampling point $\bm{z}$ is located sufficiently far away from $D$, the same argument gives
    \begin{equation*}
    \hat{\mathcal{I}}(\bz)=\varepsilon^d \sum_{j=1}^{N} \int_{B_j} \bm{f}(\by^j+\varepsilon\bm{\zeta})\,\mathrm{d}\bm{\zeta} \Big\{\mathcal{O}\Big(\frac{1}{r^{\frac{d-1}{2}}}\Big) +\mathcal{O}(\varepsilon) +\mathcal{O}\Big(\frac{1}{\mathrm{dist}(\bm{z},D)^{\frac{d-1}{2}}}\Big) \Big\},
    \end{equation*}
    as $\varepsilon\rightarrow0$, $r\rightarrow\infty$, and $\mathrm{dist}(\bm{z},D)\rightarrow \infty$.
    
    This completes the proof.
\end{proof}

\noindent
In summary, the frequency-domain direct sampling method provides a fast and reliable approach for detecting scatterers from elastic wave measurements. Its theoretical foundation, based on the asymptotic behavior of the indicator function $\hat{\mathcal{I}}(\bz)$ and the decay properties of $\mathbf{G}(\bz,\by)$, ensures that sampling points near the scatterers produce local maxima, whereas points far away yield negligible values. This makes the method suitable for non-iterative imaging of multiple scatterers in both two- and three-dimensional domains. We emphasize that the proposed method does not rely on the Helmholtz decomposition of the elastic wave field; instead, it directly uses the measured scattered wave fields for imaging.

\bibliographystyle{abbrv}
\bibliography{bibliography}

\end{document}